\documentclass[11pt]{article}

\usepackage{amssymb,amsmath}
\usepackage{amsthm}
\usepackage{hyperref}
\usepackage{mathrsfs}
\usepackage{textcomp}
\hypersetup{
    colorlinks,%
    citecolor=black,%
    filecolor=black,%
    linkcolor=black,%
    urlcolor=black
}

\usepackage{verbatim}

\usepackage{sectsty}

\makeatletter

\newdimen\bibspace
\renewenvironment{thebibliography}[1]{%
 \section*{\refname 
       \@mkboth{\MakeUppercase\refname}{\MakeUppercase\refname}}%
     \list{\@biblabel{\@arabic\c@enumiv}}%
          {\settowidth\labelwidth{\@biblabel{#1}}%
           \leftmargin\labelwidth
           \advance\leftmargin\labelsep
           \itemsep\bibspace
           \parsep\z@skip     %
           \@openbib@code
           \usecounter{enumiv}%
           \let\p@enumiv\@empty
           \renewcommand\theenumiv{\@arabic\c@enumiv}}%
     \sloppy\clubpenalty4000\widowpenalty4000%
     \sfcode`\.\@m}
    {\def\@noitemerr
      {\@latex@warning{Empty `thebibliography' environment}}%
     \endlist}

\makeatother

\makeatletter

\newtheorem{thm}{Theorem}[section]
\newtheorem{lem}[thm]{Lemma}
\newtheorem{prop}[thm]{Proposition}
\newtheorem{defn}[thm]{Definition}

\def\XXint#1#2#3{{\setbox0=\hbox{$#1{#2#3}{\int}$}
  \vcenter{\hbox{$#2#3$}}\kern-.5\wd0}}

\newcommand{\al}{\alpha}                \newcommand{\lda}{\lambda}
\newcommand{\om}{\Omega}                \newcommand{\pa}{\partial}
\newcommand{\va}{\varepsilon}           
\newcommand{\be}{\begin{equation}}      \newcommand{\ee}{\end{equation}}
                 
              \newcommand{\B}{\mathcal{B}}
\newcommand{\R}{\mathbb{R}}

\newcommand{\dlim}{\displaystyle\lim}

\begin{document}
\title{\bf\Large Liouville theorem and isolated singularity of fractional Laplacian system with critical exponents
\footnotetext{\hspace{-0.35cm} J. Bao
\endgraf jgbao@bnu.edu.cn
\vspace{0.25cm}
\endgraf L. Yi
\endgraf lym@mail.bnu.edu.cn
}}
\vspace{0.25cm}
\author{Yimei Li,\ \ Jiguang Bao\,\footnote{The  authers are supported in part by the National Natural Science Foundation of China (11631002).}\\
\small School of Mathematical Sciences, Beijing Normal University,\\
\small Laboratory of Mathematics and Complex Systems, Ministry of Education, Beijing 100875, China}
\date{}
\maketitle

\vspace{-0.8cm}

\date{}

\maketitle
{\bf Abstract}
This paper is devoted to the  fractional Laplacian system  with critical exponents. We use the method of moving sphere to derive a Liouville Theorem, and then prove the solutions in $\R^n\backslash\{0\}$ are radially symmetric and monotonically decreasing radially. Together with blow up analysis and the Pohozaev integral, we get the upper and lower bound of the local solutions in $B_1\backslash\{0\}$.  Our results is an extension of the classical work by Caffarelli et al \cite{CGS,CJSX}, Chen et al \cite{CZL} .



\section{Introduction}
The semilinear elliptic equation
\be\label{I1}
\begin{cases}
\begin{aligned}
&-\Delta u=u^{\frac{n+2}{n-2}}\quad\mbox{in }\  \R^n,\\
&u>0\quad \mbox{and}\quad u\in C^2(\R^{n}),
\end{aligned}
\end{cases}
\ee
with critical exponent has been studied in many papers, where $\Delta:=\sum_{i=1}^n\frac{\partial^2}{\partial x^2_i}$ denotes the Laplacian and  $n\geq3$. It is known that there exists a positive constant $\varepsilon$ and $y\in \R^n$ such that a $C^2$ solution of \eqref{I1} has to be  the form
\[
(n(n-2))^{\frac{n-2}{4}}\left(\frac{\varepsilon}{\varepsilon^2+|x-y|^2}\right)^{\frac{n-2}{2}}.
\]
The  celebrated Liouville-type theorem was established by Caffarelli-Gidas-Spruck \cite{CGS} and the proof was by  the method of moving planes.

Under the additional hypothesis $u(x) = O(|x| ^{2-n})$ for large $|x| $, the result was obtained earlier by Obata \cite{O} and Gidas-Ni-Nirenberg \cite{GNN}. The proof of Obata was more geometric, while the proof of Gidas-Ni-Nirenberg was by the method of moving planes.  Li-Zhang \cite{LZ} developed a rather systematic, and simpler approach to Liouville-type theorems using the method of moving sphere. They can catch the form of solutions directly, instead of reducing it to the radial symmetry of $u$ and concluding by using ODE.

Such Liouville-type theorems have played a fundamental role in the study of semilinear elliptic equations
with critical exponent, which include the Yamabe problem and the Nirenberg problem. In view of conformal geometry, a solution $u$ of \eqref{I1} defines a conformally flat metric $g_{ij} = u^{\frac{4}{n-2}}\delta_{ij}$ with constant scalar curvature.

The classical work by Schoen and Yau \cite{Schoen1, Schoen2, Schoen3} on conformally flat manifolds and the Yamabe problem has highlighted the importance of studying solutions of \eqref{I1} with a nonempty singular set. The simplest case is the following equation
\be\label{I82}
\begin{cases}
\begin{aligned}
&-\Delta u=u^{\frac{n+2}{n-2}}\quad\mbox{in }\  \R^n\backslash\{0\},\\
&u>0\quad \mbox{and}\quad u\in C^2(\R^{n}\backslash\{0\}).
\end{aligned}
\end{cases}
\ee
The issues related to \eqref{I82} have received great interest and have been widely studied in \cite{CGS,Chen1,Chen2,Chen3,Fowler,KMP,Lin1,TZ} and references therein. In particular, Fowler \cite{Fowler}
described all radial solutions of \eqref{I82}, and Caffarelli-Gidas-Spruck \cite{CGS}  proved the radial
symmetry of all solutions of \eqref{I82}.

Remark that the equation \eqref{I82} in a punctured ball,
\be\label{I3}
\begin{cases}
\begin{aligned}
&-\Delta u=u^{\frac{n+2}{n-2}}\quad\mbox{in }\  B_1\backslash\{0\},\\
&u>0\quad \mbox{and}\quad u\in C^2(B_1\backslash\{0\})
\end{aligned}
\end{cases}
\ee
has been well investigated in the cerebrated paper \cite{CGS}. Caffarelli-Gidas-Spruck \cite{CGS} proved that if $0$ is a non-removable singular point of $u$, that is, $u$ can not be extended as a continuous function to the origin, then $u$ is asymptotically symmetric with respect to the origin and furthermore,
\be\label{I10}
 u(x)=u_0(x)(1+o(1))\quad\mbox{near }\ \ x=0,
\ee
where $u_0(x)$ is a positive solution of \eqref{I82}. Thus, a corollary of \eqref{I10} is that there exist two positive constants $c_1$ and $c_2$ such that
\[
c_1|x|^{-\frac{n-2}{2}}\le u(x)\le c_2|x|^{-\frac{n-2}{2}}\quad\mbox{near }\ \ x=0.
\]

On the other hand,  coupled system with critical exponent have received ever-increasing interest and have been studied intensively in the
literature. In particular, Guo-Liu  \cite{GL} and Chen-Li \cite{CLi} independently studied the Liouville-type theorem for the following two-coupled elliptic system
\be\label{I33}
\begin{cases}
\begin{aligned}
&-\Delta u=\alpha_1u^{\frac{n+2}{n-2}}+\beta u^{\frac{2}{n-2}}v^{\frac{n}{n-2}} &\quad&\mbox{in }\ \ \R^{n},\\
&-\Delta v=\alpha_2v^{\frac{n+2}{n-2}}+\beta v^{\frac{2}{n-2}}u^{\frac{n}{n-2\sigma}} &\quad&\mbox{in }\ \ \R^{n},\\
&u, v >0 \ \ \mbox{and} \ \ u,v\in C^2(\R^{n}),
\end{aligned}
\end{cases}
\ee
where $\alpha_1$, $\alpha_2$ and $\beta$ are positive constants. They obtained that both $u$ and $v$ are radially symmetric with respect to the same point. Moreover, $(u,v)=(kU,lU)$,~where $k,l>0$ satisfies
\[
\alpha_1k^{\frac{4}{n-2}}+\beta k^{\frac{4-n}{n-2}}l^{\frac{n}{n-2}}=1,\ \
\alpha_2l^{\frac{4}{n-2}}+\beta l^{\frac{4-n}{n-2}}k^{\frac{n}{n-2}}=1,
\]
and $U$ is an entire positive solution of \eqref{I1}. In fact, they studied more general systems and \eqref{I33} is a special case of their problems.

Furthermore, the properties of positive singular solutions to a two-coupled elliptic system with critical exponents are related to coupled nonlinear Schr\"{o}dinger equations with critical exponents for nonlinear optics and Bose-Einstein condensates. Stimulated by this, Chen-Lin \cite{CZL} studied the system
\be\label{I4}
\begin{cases}
\begin{aligned}
&-\Delta u=\alpha_1u^{\frac{n+2}{n-2}}+\beta u^{\frac{2}{n-2}}v^{\frac{n}{n-2}} &\quad&\mbox{in }\ \ \R^{n}\backslash\{0\},\\
&-\Delta v=\alpha_2v^{\frac{n+2\sigma}{n-2}}+\beta v^{\frac{2}{n-2}}u^{\frac{n}{n-2\sigma}} &\quad&\mbox{in }\ \ \R^{n}\backslash\{0\},\\
&u, v>0 \ \ \mbox{and} \ \ u,v\in C^2(\R^{n}\backslash\{0\}).
\end{aligned}
\end{cases}
\ee
They proved that both $u$ and $v$ are radially symmetric about the origin and are strictly decreasing with respect to $r = |x| > 0$. With some additional conditions, they also obtained that  either $(u,v)$ can be extended as a continuous function to $0$, or there exist two positive constants $c_1$ and $c_2$ such that
\[
c_1|x|^{-\frac{n-2}{2}}\le u(x)+v(x)\le c_2|x|^{-\frac{n-2}{2}}\quad\mbox{near }\ \ x=0.
\]

In recent years, the fractional Laplacian $(-\Delta)^\sigma$ has more and more applications in Physics,
Chemistry, Biology, Probability, and Finance, and  has drawn more and more attention
from the mathematical community.
The fractional Laplacian can be understood as the infinitesimal generator of a stable
L\'{e}vy process \cite{Bertoin}. In particular, the fractional Laplacian with the critical  exponent arises in contexts such as the Euler-Lagrangian equations of Sobolev inequalities \cite{CLiB, Li04,Lie83}, a fractional Yamabe problem \cite{GMS, GQ, JX1}, a fractional Nirenberg problem \cite{JLX, JLX2} and so on.

The fractional Laplacian  takes the form
\begin{equation}\label{Wthe definition of fractional laplacian}
(-\Delta)^{\sigma}u(x):=C_{n,\sigma}\lim_{\varepsilon\rightarrow0^{+}}\int_{\R^n\setminus B_{\varepsilon}(x)}\frac{u(x)-u(y)}{|x-y|^{n+2\sigma}}dy,
\end{equation}
and $$C_{n,\sigma}:=\frac{2^{2\sigma}\sigma\Gamma(\frac{n}{2}+\sigma)}{\pi^{\frac{n}{2}}\Gamma(1-\sigma)}$$ with the gamma function $\Gamma$. The operator $(-\Delta)^\sigma $ is well defined in  the Schwartz space of rapidly decaying $C^\infty$ functions in $\R^n$.

One can also define the fractional Laplacian acting on spaces of functions with weaker regularity. Considering the space
\[
L_{\sigma}(\R^n):=\left\{u\in L^1_{{\rm loc}}(\R^n): \int_{\R^n}\frac{|u(x)|}{1+|x|^{n+2\sigma}}dx<\infty\right\},
\]
endowed with the norm
\[
\|u\|_{L_{\sigma}(\R^n)}:=\int_{\R^n}\frac{|u(x)|}{1+|x|^{n+2\sigma}}dx.
\]
We can verify that if $u\in C^{2}(\R^{n})\cap L_{\sigma}(\R^n)$, the integral on the right hand side of
\eqref{Wthe definition of fractional laplacian} is well defined in $\R^n$. Moreover, from \cite[Proposition 2.4]{Silvestre}, we have
\[
\begin{array}{ll}
(-\Delta)^\sigma u\in C^{1,1-2\sigma}(\R^n),\quad\quad\mbox{if}\quad0<\sigma<1/2,\vspace{0.2cm}\\
(-\Delta)^\sigma u\in C^{0,2-2\sigma}(\R^n),\quad\quad\mbox{if}\quad 1/2\leq\sigma< 1.
\end{array}
\]

Jin-Li-Xiong \cite{JLX} studied the Liouville-type theorem  for the equation
\be\label{I5}
\begin{cases}
\begin{aligned}
&(-\Delta)^\sigma u=u^{\frac{n+2\sigma}{n-2\sigma}}\quad\mbox{in }\  \R^n,\\
&u>0\quad \mbox{and}\quad u\in C^{2}(\R^{n})\cap L_{\sigma}(\R^n).
\end{aligned}
\end{cases}
\ee
A feature of \eqref{I5} is  conformal invariant, and one may refer to \cite{CG,GZ} for its connections to conformal geometry. Since the radial symmetry property is essential for the development of symmetrization techniques for fractional elliptic and parabolic partial differential equations, a lot of  people are interested in the radial symmetry results. Jin-Li-Xiong \cite{JLX} obtained that there exist two  positive constants $\varepsilon$, $\varepsilon_1$ and $y\in \R^n$ such that a solution of \eqref{I5} has to be  the form
\[
\left(\frac{\varepsilon_1}{\varepsilon^2+|x-y|^2}\right)^{\frac{n-2\sigma}{2}}.
\]
For more details about the Liouville  Theorem, please see \cite{CLiB,CHEN1}  and the references therein.

Caffarelli-Jin-Sire-Xiong \cite{CJSX} studied the global behaviors of positive solutions of the fractional Yamabe equations
\be\label{I6}
\begin{cases}
\begin{aligned}
&(-\Delta)^\sigma u=u^{\frac{n+2\sigma}{n-2\sigma}}\quad\mbox{in }\  \R^n\backslash \{0\},\\
&u>0\quad \mbox{and}\quad u\in C^2(\R^{n}\backslash \{0\})\cap L_{\sigma}(\R^n),
\end{aligned}
\end{cases}
\ee
with an isolated singularity at the origin. They proved that if the origin is non-removable isolated singularity, then the solution $u$ of \eqref{I6} is radial symmetric with respect to the origin and strictly decreasing with respect to $r = |x| > 0$.  It is consistent with the result of Caffarelli-Gidas-Spruck \cite{CGS}  on Laplacian. Jin-de Queiroz-Sire-Xiong \cite{JQSX} obtained the same result if the equation \eqref{I6} is defined in $\R^n\setminus \R^k$ $(1\leq k\leq n-2\sigma)$ and  there exists  $x_0\in\R^k$ such that $\limsup_{x\rightarrow (x_0, 0, \cdots, 0)}u=+\infty$.

Caffarelli-Jin-Sire-Xiong \cite{CJSX} also studied the local behaviors of positive solutions of the fractional Yamabe equations
\be\label{I7}
\begin{cases}
\begin{aligned}
&(-\Delta)^\sigma u=u^{\frac{n+2\sigma}{n-2\sigma}}\quad\mbox{in }\  B_1\backslash \{0\},\\
&u>0\quad \mbox{and}\quad u\in C^2(B_1\backslash \{0\})\cap L_{\sigma}(\R^n),
\end{aligned}
\end{cases}
\ee
with an isolated singularity at the origin. They  obtained that  either $u$ can be extended as a continuous function to $0$, or there exist two positive constants $c_1$ and $c_2$ such that
\[
c_1|x|^{-\frac{n-2\sigma}{2}}\le u(x)\leq c_2|x|^{-\frac{n-2\sigma}{2}}\quad\mbox{near }\ \ x=0.
\]

Inspired by the work on Laplacian, the Liouville Theorem for the system
\be\label{I8}
\begin{cases}
\begin{aligned}
&(-\Delta)^{\sigma}u=\alpha_1u^{\frac{n+2\sigma}{n-2\sigma}}+\beta u^{\frac{2\sigma}{n-2\sigma}}v^{\frac{n}{n-2\sigma}} &\quad&\mbox{in }\ \ \R^{n},\\
&(-\Delta)^{\sigma}v=\alpha_2v^{\frac{n+2\sigma}{n-2\sigma}}+\beta v^{\frac{2\sigma}{n-2\sigma}}u^{\frac{n}{n-2\sigma}} &\quad&\mbox{in }\ \ \R^{n},\\
&u, v >0 \ \ \mbox{and} \ \ u,v\in C^2(\R^{n})\cap L_{\sigma}(\R^{n}),
\end{aligned}
\end{cases}
\ee
where $n\geq 3$, has been studied in \cite{LL,ZC}. By the method of moving plane in the integral form, \cite{ZC} obtained that the positive solution of \eqref{I8} is radial symmetry. But they do not give the form of the solution of the system.

In this paper, using the method of moving sphere and some Calculus propositions, we  catch the form of the solution, which is consistent with the work of Chen-Li \cite{CLi} and Guo-Liu \cite{GL} on Laplacian.

\begin{thm}\label{thm1}
Let $(u,v)$ be a solution of \eqref{I8},
then both $u$ and $v$ are radially symmetric with respect to the same point. In particular, $(u,v)=(k\widehat{U},l\widehat{U})$,~where $k,l>0$ satisfies
\[
\alpha_1k^{\frac{4\sigma}{n-2\sigma}}+\beta k^{\frac{4\sigma-n}{n-2\sigma}}l^{\frac{n}{n-2\sigma}}=1,\ \
\alpha_2l^{\frac{4\sigma}{n-2\sigma}}+\beta l^{\frac{4\sigma-n}{n-2\sigma}}k^{\frac{n}{n-2\sigma}}=1,
\]
and $\widehat{U}$ is a solution of \eqref{I5}.
\end{thm}
For $x\in \R^n$ and $\lambda>0$, define
\[
u_{ x,\lda}(y):= \left(\frac{\lda }{|y-x|}\right)^{n-2\sigma}u\left(x+\frac{\lda^2(y-x)}{|y-x|^2}\right)\quad\mbox{in}\ \ \R^n\backslash\{x\},
\]
the Kelvin transformation of $u$  with respect to the ball $B_{\lambda}(x)$. If $(u,v)$ is a solution of \eqref{I8}, then $(u_{ x,\lda},v_{ x,\lda})$ is a solution of \eqref{I8} in the corresponding domain. Such conformal invariance allows us to use the moving sphere method introduced by Li-Zhu \cite{LiZhu}. This observation has also been used in \cite{CJSX,JQSX,JLX}.

Next, by the similar method  we are going to prove the radial symmetry of positive singular solutions, which is an extension of  Chen-Lin \cite{CZL} work on Laplacian.
\begin{thm}\label{ZHUYAO}
Let $(u,v)$ be a singular solution  of
\be\label{1}
\begin{cases}
\begin{aligned}
&(-\Delta)^{\sigma}u=\alpha_1u^{\frac{n+2\sigma}{n-2\sigma}}+\beta u^{\frac{2\sigma}{n-2\sigma}}v^{\frac{n}{n-2\sigma}} &\quad&\mbox{\rm in }\ \ \R^{n}\backslash\{0\},\\
&(-\Delta)^{\sigma}v=\alpha_2v^{\frac{n+2\sigma}{n-2\sigma}}+\beta v^{\frac{2\sigma}{n-2\sigma}}u^{\frac{n}{n-2\sigma}} &\quad&\mbox{\rm in }\ \ \R^{n}\backslash\{0\},\\
&u, v >0 \ \ \mbox{and} \ \ u,v\in C^2(\R^{n}\backslash\{0\})\cap L_{\sigma}(\R^{n}),
\end{aligned}
\end{cases}
\ee
that is, $\limsup_{x\rightarrow 0}u+\limsup_{x\rightarrow 0}v=\infty$. Then both $u$ and $v$ are radially symmetric and monotonically decreasing radially.
\end{thm}
More interesting, we also study the local behaviors of positive solutions of system in a punctured ball
\be\label{91}
\begin{cases}
\begin{aligned}
&(-\Delta)^{\sigma}u=\alpha_1u^{\frac{n+2\sigma}{n-2\sigma}}+\beta u^{\frac{2\sigma}{n-2\sigma}}v^{\frac{n}{n-2\sigma}} &\quad&\mbox{in }\ \ B_1\backslash\{0\},\\
&(-\Delta)^{\sigma}v=\alpha_2v^{\frac{n+2\sigma}{n-2\sigma}}+\beta v^{\frac{2\sigma}{n-2\sigma}}u^{\frac{n}{n-2\sigma}} &\quad&\mbox{in }\ \ B_1\backslash\{0\},\\
&u, v >0 \ \ \mbox{and} \ \ u,v\in  C^2(B_1\backslash\{0\})\cap L_{\sigma}(\R^{n}).
\end{aligned}
\end{cases}
\ee
\begin{thm}\label{thm:a}
Let $(u,v)$ be a  solution of \eqref{91}. Then either $(u,v)$ can be extended as a continuous function near $0$, or there exist two positive constants $c_1$ and $c_2$ such that
\[
c_1|x|^{-\frac{n-2\sigma}{2}}\le (u+v)(x)\le c_2|x|^{-\frac{n-2\sigma}{2}}\quad\quad\mbox{\rm near}\ \ x=0.
\]
\end{thm}
Our paper is organized as follows. Section \ref{F1} includes  some definitions of basic space and  elementary propositions which will be used in our following proof. Section \ref{F4} is devoted to obtain the Liouville Theorem, that is Theorem \ref{thm1}. Theorem \ref{ZHUYAO} on  symmetry of global solutions of \eqref{1} is proved in Section \ref{F3}. In section \ref{F2},  we obtain the blow up upper and lower bounds and prove Theorem \ref{thm:a}. Finally, we collect some propositions in Section \ref{F5}.
\section{Preliminaries}\label{F1}
\subsection{The Extension Method}
Since the operator $(-\Delta)^\sigma $ is nonlocal, the traditional methods on local differential operators, such as on Laplacian  may not work on this nonlocal operator. To circumvent this difficulty, Caffarelli and Silvestre \cite{CS} introduced the extension method that reduced this nonlocal problem into a local one in higher dimensions with the conormal derivative boundary condition.

More precisely,  for $u\in C^{2}(\R^n)\cap L_{\sigma}(\R^n)$,  define
\be\label{DD}
U(x,t):=\int_{\R^n}\mathcal{P}_{\sigma}(x-\xi,t)u(\xi)d\xi,\quad 
\ee
where
\[
\mathcal{P}_{\sigma}(x,t):=\frac{\beta(n,\sigma)t^{2\sigma}}{(|x|^2+t^2)^{(n+2\sigma)/2}}
\]
with a constant $\beta(n,\sigma)$ such that $\int_{\R^n}\mathcal{P}_{\sigma}(x,1)dx=1$. 
It follows that
$$U\in C^{2}(\R^{n+1}_{+})\cap C(\overline{\R^{n+1}_{+}}),\quad t^{1-2\sigma} \pa_t U(x,t)\in C(\overline{\R^{n+1}_{+}}),$$
and $U$ satisfies
\be\label{BI}
\mathrm{div}(t^{1-2\sigma} \nabla U)=0 \quad\mbox{in }\ \ \R^{n+1}_{+}
\ee
\[
U=u\quad\quad\quad\ \mbox{on }\ \ \partial'\R^{n+1}_{+},
\]
where $\partial'\R^{n+1}_{+}=\R^{n}\times\{0\}$.

In addition, by works of Caffarelli and Silvestre \cite{CS}, it is known that up to a constant,
$$\frac{\partial U}{\partial \nu^\sigma}=(-\Delta)^{\sigma}u\quad\mbox{on } \ \partial'\R^{n+1}_{+},$$
where
\[
\frac{\partial U(x,0)}{\partial \nu^\sigma}:=-\lim_{t\to 0^+} t^{1-2\sigma} \pa_t U(x,t).
\]
From this and  $(u,v)$ is a solution of \eqref{I8}, we have
\be\label{Weq:ex0001}
\frac{\partial U}{\partial \nu^\sigma}=\alpha_1u^{\frac{n+2\sigma}{n-2\sigma}}+\beta u^{\frac{2\sigma}{n-2\sigma}}v^{\frac{n}{n-2\sigma}} \quad\mbox{ on}\ \ \partial'\R^{n+1}_{+}.
\ee
In order to study the behavior of the solution \eqref{I8}, we just need to study the behaviors of $U$ defined by \eqref{DD}.
\subsection{A Weight Sobolev Space}
In our paper, the solutions of \eqref{I8}, \eqref{1} and \eqref{91} $(u,v)$ are understood in the classical sense. By the above argument, it follows that $(U,V)$ are also understood in the classical sense. We need some useful propositions such as the local max/min principle in our proof. There is no need to request  classical solutions to guarantee these propositions that are right. That is, the weak solutions is enough. Hence, we introduce the Weight Sobolev Space and the definition of weak solutions.

Let $D$ be an open set in $\R_+^{n+1}$. Denote by $L^2(t^{1-2\sigma}, D)$ the Banach space of all measurable functions $U$, defined on $D$, for which
\[
\|U\|_{L^2(t^{1-2\sigma},D)}:=\left(\int_{D}t^{1-2\sigma} |U|^2dX\right)^{\frac{1}{2}}<\infty,
\]
and $X:=(x,t)\in \R^{n}\times \R_+$. We say that $U\in W^{1,2}(t^{1-2\sigma},D)$ if $U\in L^2(t^{1-2\sigma}, D)$, and its weak derivatives $\nabla U$ exist and belong to $L^2(t^{1-2\sigma}, D)$. The norm of $U$ in $W^{1,2}(t^{1-2\sigma}, D)$ is given by
\[
\|U\|_{W^{1,2}(t^{1-2\sigma}, D)}:=\left(\int_{D}t^{1-2\sigma}|U|^2dX+\int_{D}t^{1-2\sigma}|\nabla U|^2dX\right)^{\frac{1}{2}}.
\]
Notice that $C^\infty (D)$ is dense in $W^{1,2}(t^{1-2\sigma},D)$. Moreover, if $D$ is a bounded domain with Lipschitz boundary, then there exists a bounded linear extension operator from $W^{1,2}(t^{1-2\sigma},D)$ to $W^{1,2}(t^{1-2\sigma},\R^{n+1}_+)$.

Next, we present a  well known result called the trace inequality.
\begin{prop}\label{tra}{\rm \cite[Proposition 2.1]{JLX}}
If $U\in W^{1,2}(t^{1-2\sigma}, \R^{n+1}_+)$, then there exists a positive constant $C$  depending only on $n$ and $\sigma$ such that
\be
\bigg(\int_{\R^n}|U(\cdot,0)|^{\frac{2n}{n-2\sigma}}dx\bigg)^{\frac{n-2\sigma}{2n}}\leq C\bigg(\int_{\R^{n+1}_{+}}t^{1-2\sigma}|\nabla U|^{2}dxdt\bigg)^{\frac{1}{2}}
\ee
\end{prop}
We denote $\B_R(X)$ as the ball in $\R^{n+1}$ with radius $R$ and center $X$, $\B^+_R(X)$ as $\B_R(X)\cap \R^{n+1}_+$, and $B_R(x)$ as the ball in $\R^{n}$ with radius $R$ and center $x$. We also write $\B_R(0)$, $\B^+_R(0)$, $B_R(0)$ as $\B_R$, $\B_R^+$, $B_R$ for short respectively. For a domain $D\subset \R^{n+1}_+$ with boundary $\pa D$, we denote $\pa' D:=\partial D\cap\partial\R^{n+1}_+$ and $\pa''D:=\pa D \cap\R^{n+1}_+$. It is easy to see that  $\pa' \B_R^+(X):=\partial\B^+_R(X)\cap\partial\R^{n+1}_+$ and $\pa'' \B_{R}^+(X):=\pa \B^+_{R}(X)\cap\R^{n+1}_+$.

\begin{defn}
Let $\pa' D\neq\emptyset$, $a\in L_{\rm loc}^{2n/(n+2\sigma)}(\pa' D)$ and $b\in L_{\rm loc}^{1}(\pa' D)$. We say $U\in W^{1,2}(t^{1-2\sigma}, D)$ is a weak solution (resp. supersolution, subsolution) of
\[
\begin{cases}
\begin{aligned}
&\mathrm{div}(t^{1-2\sigma} \nabla U)=0 &\quad&\mbox{\rm in }\ \ D,\\
&\frac{\partial U}{\partial \nu^\sigma}=aU+b&\quad&\mbox{\rm on }\  \pa' D,
\end{aligned}
\end{cases}
\]
if for every nonnegative $\Phi\in C^\infty_c(D\cup \pa' D)$,
\[
\int_{\B^{+}_{R}}t^{1-2\sigma}\nabla U\nabla \Phi dX=({\rm{resp}}. \geq,\leq )\int_{B_{R}}(aU+b)\Phi dx.
\]
\end{defn}

\subsection{Local max/min principle}
\begin{prop}\label{Harnack inequality}
Let $U$, $V\in W^{1,2}(t^{1-2\sigma}, \B^{+}_{R})$ be a nonnegative weak subsolution ( resp.supersolution,solution)  of
\be\label{HA}
\begin{cases}
\begin{aligned}
&\mathrm{div}(t^{1-2\sigma} \nabla U)\geq(\rm resp.\leq,=)0 &\quad&\mbox{\rm in }\ \ \B^{+}_{R},\\
&\mathrm{div}(t^{1-2\sigma} \nabla V)\geq(\rm resp.\leq,=)0 &\quad&\mbox{\rm in }\ \ \B^{+}_{R},\\
&\frac{\partial U}{\partial \nu^\sigma}\leq(\rm resp.\geq,=)\alpha_1U^{\frac{n+2\sigma}{n-2\sigma}}+\beta U^{\frac{2\sigma}{n-2\sigma}}V^{\frac{n}{n-2\sigma}} &\quad&\mbox{\rm on }\  \pa'\B^{+}_{R},\\
&\frac{\partial V}{\partial \nu^\sigma}\leq(\rm resp.\geq,=)\alpha_2V^{\frac{n+2\sigma}{n-2\sigma}}+\beta V^{\frac{2\sigma}{n-2\sigma}}U^{\frac{n}{n-2\sigma}} &\quad&\mbox{\rm on }\  \pa'\B^{+}_{R}.
\end{aligned}
\end{cases}
\ee
Then
\begin{description}
  \item[(i)] If $U$, $V\in W^{1,2}(t^{1-2\sigma}, D)$ is a weak subsolution of \eqref{HA}, then for all $p>0$,
\[
\sup_{\B^+_{R/2}}(U+V)\le C \|U+V\|_{L^p(t^{1-2\sigma},\B^+_{3R/4})},
\]
where $C$ is a positive constant and depends only on $n$, $\sigma$, $R$, $p$, $\|U+V\|_ {W^{1,2}(t^{1-2\sigma}, \B^{+}_{R})}$.
  \item[(ii)] If $U$, $V\in W^{1,2}(t^{1-2\sigma}, D)$ is a weak supersolution of \eqref{HA}, then for all $0<p_0\leq (n+1)/n$,
  \[
\inf_{\B^+_{R/2}}(U+V)\ge C \|U+V\|_{L^{p_0}(t^{1-2\sigma},\B^+_{3R/4})}.
\]
where $C$ is a positive constant and depends only on $n$, $\sigma$, $R$, $p_0$, $\|U+V\|_ {W^{1,2}(t^{1-2\sigma}, \B^{+}_{R})}$.
  \item[(iii)] If $U$, $V\in W^{1,2}(t^{1-2\sigma}, D)$ is a weak solution of \eqref{HA}, then there exists $\alpha\in (0,1)$ such that
\[
\|U\|_{ C^\alpha(\overline{\B^+_{R/2}})}+\| V\|_{ C^\alpha(\overline{\B^+_{R/2}})}\leq C\| U+V\|_{ L^\infty(\B^+_{R/2})},
\]
where $C$ is a positive constant and depends only on $n$, $\sigma$, $R$, $\|U+V\|_ {W^{1,2}(t^{1-2\sigma}, \B^{+}_{R})}$.
\end{description}
\end{prop}
\begin{proof}
We define $u(x):=U(x,0)$, $v(x):=V(x,0)$. Via a straightforward calculation, we have
\begin{equation}\label{DAXIAO}
\begin{split}
&\alpha_1u^{\frac{n+2\sigma}{n-2\sigma}}+\beta u^{\frac{2\sigma}{n-2\sigma}}v^{\frac{n}{n-2\sigma}}+\alpha_2v^{\frac{n+2\sigma}{n-2\sigma}}+\beta v^{\frac{2\sigma}{n-2\sigma}}u^{\frac{n}{n-2\sigma}}\\
\leq&\alpha_1u^{\frac{4\sigma}{n-2\sigma}}(u+v)+\beta u^{\frac{2\sigma}{n-2\sigma}}v^{\frac{2\sigma}{n-2\sigma}}v+\alpha_2v^{\frac{4\sigma}{n-2\sigma}}(u+v)+\beta v^{\frac{2\sigma}{n-2\sigma}}u^{\frac{2\sigma}{n-2\sigma}}u\\
=&(\alpha_1u^{\frac{4\sigma}{n-2\sigma}}+\beta u^{\frac{2\sigma}{n-2\sigma}}v^{\frac{2\sigma}{n-2\sigma}}+\alpha_2v^{\frac{4\sigma}{n-2\sigma}})(u+v).
\end{split}
\end{equation}
It follows that if $(U,V)$ is a weak subsolution of \eqref{HA}, then $U+V$ is a weak subsolution of
\[
\begin{cases}
\begin{aligned}
&\mathrm{div}(t^{1-2\sigma} \nabla (U+V))\geq0 &\quad&\mbox{in }\ \ \B^{+}_{R},\\
&\frac{\partial (U+V)}{\partial \nu^\sigma}\leq(\alpha_1u^{\frac{4\sigma}{n-2\sigma}}+\beta u^{\frac{2\sigma}{n-2\sigma}}v^{\frac{2\sigma}{n-2\sigma}}+\alpha_2v^{\frac{4\sigma}{n-2\sigma}})(u+v) &\quad&\mbox{on }\  \pa'\B^{+}_{R},
\end{aligned}
\end{cases}
\]
Combining with the trace inequality (Proposition \ref{tra}), we have $u$, $v\in L^{\frac{2n}{n-2\sigma}} (B_{R})$ . As a result,
\[
\alpha_1u^{\frac{4\sigma}{n-2\sigma}}+\beta u^{\frac{2\sigma}{n-2\sigma}}v^{\frac{2\sigma}{n-2\sigma}}+\alpha_2v^{\frac{4\sigma}{n-2\sigma}}\in L^{\frac{n}{2\sigma}} (B_{R}),
\]
Then by \cite[Lemma 2.8]{JLX}, for $0<r<R$, we have
\[
\|u+v\|_{ L^{p} (B_{r})}\le C \|U+V\|_{W(t^{1-2\sigma},\B^+_{R})},
\]
where $C$ is a positive constant and  depends only on $n$, $\sigma$, $r$,  $\|u\|_ {L^{\frac{2n}{n-2\sigma}} (B_{R})}$, $\|v\|_ {L^{\frac{2n}{n-2\sigma}} (B_{R})}$ and  $p=\frac{2(n+1)}{n-2\sigma}$.
As an easy consequence of  $u>0$, $v>0$, we have
\[
u\in L^{p} (B_{r}),\quad v\in L^{p} (B_{r}).
\]
It follows that
\[
u^{\frac{4\sigma}{n-2\sigma}},\ v^{\frac{4\sigma}{n-2\sigma}},\ u^{\frac{2\sigma}{n-2\sigma}}v^{\frac{n}{n-2\sigma}}\in L^{q} (B_{r}),
\]
where $q>\frac{n}{2\sigma}$ and $0<r<R$.

%
%
Then by \cite[Proposition 2.6]{JLX}, we conclude that for any $p>0$,
\[
\sup_{\B^+_{R/2}}(U+V)\le C \|U+V\|_{L^p(t^{1-2\sigma},\B^+_{3R/4})}.
\]
Here $C$ depends only on $n$, $\sigma$, $R$, $\|u^{\frac{4\sigma}{n-2\sigma}}\|_ {L^q (B_{R})}$, $\|v^{\frac{4\sigma}{n-2\sigma}}\|_ {L^q (B_{R})}$,
$\|u^{\frac{2\sigma}{n-2\sigma}}v^{\frac{n}{n-2\sigma}}\|_{ L^{q} (B_{r})}$.

On the other hand, as the above argument, if $(U,V)$ is a weak supersolution of \eqref{HA}, then $U+V$ is a weak supersolution of
\[
\begin{cases}
\begin{aligned}
&\mathrm{div}(t^{1-2\sigma} \nabla (U+V))=0 &\quad&\mbox{in }\ \ \B^{+}_{R},\\
&\frac{\partial (U+V)}{\partial \nu^\sigma}\geq\beta u^{\frac{2\sigma}{n-2\sigma}}v^{\frac{2\sigma}{n-2\sigma}}(u+v) &\quad&\mbox{on }\  \pa'\B^{+}_{R},
\end{aligned}
\end{cases}
\]
and
\[
\beta u^{\frac{2\sigma}{n-2\sigma}}v^{\frac{2\sigma}{n-2\sigma}}\in L^q (B_{R}),
\]
With the help of  \cite[Proposition 2.6]{JLX}, we obtain that for all $0<p_0\leq (n+1)/n$
\[
\inf_{\B^+_{R/2}}(U+V)\ge C \|U+V\|_{L^{p_0}(t^{1-2\sigma},\B^+_{3R/4})},
\]
Here $C$ depends only on $n$, $\sigma$, $R$,  $\|u^{\frac{2\sigma}{n-2\sigma}}v^{\frac{2\sigma}{n-2\sigma}}\|_{L^q (B_{R})}$.

Consider
\[
\begin{cases}
\begin{aligned}
&\mathrm{div}(t^{1-2\sigma} \nabla U)=0 &\quad&\mbox{in }\ \ \B^{+}_{R},\\
&\frac{\partial U}{\partial \nu^\sigma}=\alpha_1u^{\frac{n+2\sigma}{n-2\sigma}}+\beta u^{\frac{2\sigma}{n-2\sigma}}v^{\frac{n}{n-2\sigma}} &\quad&\mbox{on }\ \ \pa'\B^{+}_{R},
\end{aligned}
\end{cases}
\]
and
\[
\begin{cases}
\begin{aligned}
&\mathrm{div}(t^{1-2\sigma} \nabla V)=0 &\quad&\mbox{in }\ \ \B^{+}_{R},\\
&\frac{\partial V}{\partial \nu^\sigma}=\alpha_2v^{\frac{n+2\sigma}{n-2\sigma}}+\beta v^{\frac{2\sigma}{n-2\sigma}}u^{\frac{n}{n-2\sigma}} &\quad&\mbox{on }\ \ \pa'\B^{+}_{R},
\end{aligned}
\end{cases}
\]
respectively. With the help of  \cite[Proposition 2.6]{JLX}, there exists $\al\in (0,1)$ such that
\[
\|U\|_{ C^\alpha(\overline{\B^+_{R/2}})}+\| V\|_{ C^\alpha(\overline{\B^+_{R/2}})}\leq C\| U+V\|_{ L^\infty(\B^+_{R/2})},
\]
where $C$ is a positive constant and depends only on $n$, $\sigma$, $R$, $\|U+V\|_ {W^{1,2}(t^{1-2\sigma}, \B^{+}_{R})}$.
\end{proof}

\subsection{B\^{o}cher's Theorem}
\begin{prop}\label{Bocher}
Suppose that $U\in W_{\rm loc}^{1,2}(t^{1-2\sigma}, \R^{n+1}_+)$ with $U>0$ in $\R^{n+1}_+$ is a weak solution of
\be\label{Bocher1}
\begin{cases}
\begin{aligned}
&\mathrm{div}(t^{1-2\sigma} \nabla U)=0 &\quad&\mbox{\rm in }\ \ \R^{n+1}_+,\\
&\frac{\partial U}{\partial \nu^\sigma}=0 &\quad&\mbox{\rm on }\  \partial'\R^{n+1}_{+}\backslash \{0\}.
\end{aligned}
\end{cases}
\ee
Then
\[
U(\xi)=A|\xi|^{2\sigma-n}+B,
\]
where $A$ and $B$ are  nonnegative constants.
\end{prop}
\begin{proof}
For any fixed $Y\in \R^{n+1}_{+}$ and every $R> |Y|$,  by \cite[Lemma 4.10]{JLX} we may write
\[
U(\xi)=A|\xi|^{2\sigma-n}+V(\xi)\quad\quad\mbox{in }\ \B^+_{R}\backslash\{0\},
\]
where $A$ is a   nonnegative constant and $V$ satisfies
\[
\begin{cases}
\begin{aligned}
&\mathrm{div}(t^{1-2\sigma} \nabla V)=0 &\quad&\mbox{\rm in }\ \ \B^+_{R},\\
&\frac{\partial V}{\partial \nu^\sigma}=0 &\quad&\mbox{\rm on }\  \partial'\B^+_{R}.
\end{aligned}
\end{cases}
\]
With the help of the minimum principle to $V$ on $\B^+_{R}$, we have
\[
V(Y)\geq \min\{V(\xi),\ \ |\xi|=R\}>-A|R|^{2\sigma-n},
\]
where the positive of $U$ gives the second inequality. Let $R\rightarrow +\infty$, we see that $V$ is nonnegative and satisfies
\[
\begin{cases}
\begin{aligned}
&\mathrm{div}(t^{1-2\sigma} \nabla V)=0 &\quad&\mbox{\rm in }\ \ \R^{n+1}_+,\\
&\frac{\partial V}{\partial \nu^\sigma}=0 &\quad&\mbox{\rm on }\  \partial'\R^{n+1}_{+}.
\end{aligned}
\end{cases}
\]
By Liouville theorem \cite{Bogdan}, it follows that $V$ is a nonnegative constant. This finishes the proof of this proposition.
\end{proof}
\section{Liouville Theorem}\label{F4}
For any $x\in \R^{n}$, assume that there exists  a positive constant $\lambda=\lambda(x)$  so that
\be\label{eq:bigger1}
u_{x,\lda}(y)= u(y),\ \ v_{x,\lda}(y)= v(y)\quad\mbox{in }\  \R^{n}\backslash \{x\}.
\ee
Then by Proposition \ref{CHANG1}, we complete the proof of Theorem \ref{thm1}. On the other hand, \eqref{eq:bigger1} can be reduced to
\be\label{eq:bigger2}
U_{X,\lda}(\xi)=U(\xi),\ \ V_{X,\lda}(\xi)= V(\xi)\quad\mbox{in }\  \R_+^{n+1}\backslash \{X\},
\ee
where $U$, $V$  are defined as \eqref{DD}, $X:=(x,0)$.

In order to prove \eqref{eq:bigger2}, we introduce
\[
\bar \lda(x):=\sup \big\{\mu>0\ \big|\ U_{X,\lda}(\xi)\leq U(\xi),V_{X,\lda}(\xi)\leq V(\xi)\quad\mbox{in }\  \R^{n+1}_{+}\backslash \B^+_\lda(X),~\forall~ \lambda\in(0,\mu)\big\}.
\]
First, we need the following lemma to guarantee that the set over which we are taking the supremum is non-empty and then $\bar\lambda(x)$ is well defined.
\begin{lem}\label{lemma1}
For all $x\in\R^n$ there exists $\lda_0(x)>0$ such that for all $\lda\in(0,\lda_0(x))$,
\be\label{eq:bigger}
U_{X,\lda}(\xi)\le U(\xi),\ \ V_{X,\lda}(\xi)\le V(\xi)\quad\mbox{\rm in }\ \R^{n+1}_{+}\backslash \B^+_\lda(X).
\ee
\rm{It is easy to see that either} $\bar \lda(x)<\infty$ \rm{or} $\bar \lda(x)=\infty$.
\end{lem}
Now we shall prove
\begin{lem}\label{lemma2}
If $\bar \lda(x)<+\infty$ for some $x\in \R^n$, ~then ~$U_{X,\bar\lda(x)}(\xi)=U(\xi)$,~$V_{X,\bar\lda(x)}(\xi)=V(\xi)$ in $\R^{n+1}_{+}\backslash\{X\}$.
\end{lem}
\begin{lem}\label{lemma3}
Either $\bar \lda(x)<+\infty$ for all $x\in \R^n$, ~or $\bar \lda(x)=+\infty$ for all $x\in \R^n$.
\end{lem}
\begin{lem}\label{lemma4}
If $\bar \lda(x)=+\infty$ for all $x\in \R^n$, ~then ~$U(x,t)=U(0,t)$,~$ V(x,t)=V(0,t)$ for all $(x,t)\in \R_+^{n+1}$.
\end{lem}
By Lemma \ref{lemma4}, which implies that both $u$ and $v$ are positive constants. By
\[
(-\Delta)^\sigma u=\alpha_1u^{\frac{n+2\sigma}{n-2\sigma}}+\beta u^{\frac{2\sigma}{n-2\sigma}}v^{\frac{n}{n-2\sigma}}\quad\mbox{in }\ \ \R^{n},
\]
it is clearly to see that this case never happens. With the help of the conclusion of Lemma \ref{lemma2}, we can prove Lemma \ref{lemma3}. Combining Lemma \ref{lemma4} with the above argument, we can
obtain that for all $x\in \R^n$, $\bar \lda(x)<+\infty$. Using Lemma \ref{lemma2} again, \eqref{eq:bigger1} follows.

Now, in turn, we start to prove Lemma \ref{lemma1}-\ref{lemma4} .

\begin{proof}[Proof of Lemma \ref{lemma1}] For simplicity, we prove \eqref{eq:bigger} only for $U$, since the proof on $V$ is similar. We are going to show that there exist $\mu$ and $\lambda_0(x)$ satisfying $0<\lambda_0(x)<\mu$, ~which may depend on ~$x$,~ such that for $\lda \in (0,\lambda_0(x))$,
\be\label{N}
U_{X,\lambda}(\xi)\leq U(\xi)\quad\mbox{in }\  \R^{n+1}_+\backslash\overline{\B^+_\mu(X)},
\ee
and then
\be\label{Q}
U_{X,\lda }(\xi)\leq U(\xi)\quad\mbox{in }\ \overline{\B^+_\mu(X)}\backslash\B^+_\lda(X).
\ee

As the first step, we are going to prove \eqref{N}. For every $0<\mu<10$, which may depend on $x$, we can choose \be\label{IDA}
\lambda_0(x)=\mu\min\left\{\left(\inf\limits_{\partial'' \B^+_{\mu}(X)} U/\sup\limits_{ \B^+_{\mu}(X)}U\right)^{\frac{1}{n-2\sigma}} ,\left(\inf\limits_{\partial'' \B^+_{\mu}(X)} V/\sup\limits_{ \B^+_{\mu}(X)}V\right)^{\frac{1}{n-2\sigma}} \right\}.
\ee
Let
\[
\phi(\xi):=\left(\frac{\mu}{|\xi-X|}\right)^{n-2\sigma}\inf\limits_{\partial'' \B^+_{\mu}(X)} U,
\]
which satisfies
\[
\begin{cases}
\begin{aligned}
&\mathrm{div}(t^{1-2\sigma}\nabla \phi)=0&\quad& \mbox{in }\ \R^{n+1}_+\setminus \B^+_{\mu}(X),\\
&\frac{\partial \phi}{\partial \nu^\sigma} =0&\quad& \mbox{on }\ \partial'(\R^{n+1}_+\setminus \B^+_{\mu}(X))
 \end{aligned}
\end{cases}
\]
 and
\[
\phi(\xi)=\inf\limits_{\partial'' \B^{+}_{\mu}(X)} U\leq U(\xi)\quad\mbox{on }\ \partial'' \B^{+}_{\mu}(X).
\]
It is easy to see that
\[
\lim_{\xi\rightarrow+\infty}\phi(\xi)=0\leq \lim_{\xi\rightarrow+\infty}U(\xi).
\]
By the standard maximum principle argument,  we conclude that
\be\label{9}
\left(\frac{\mu}{|\xi-X|}\right)^{n-2\sigma}\inf_{\partial'' \B^{+}_{\mu}(X)} U\leq U(\xi)\quad\mbox{in }\ \R^{n+1}_+\setminus \B^+_{\mu}(X).
\ee
Then for all $\xi\in\R^{n+1}_+\setminus \B^+_{\mu}(X)$ and $\lambda\in(0,\lda_0(x))$, it follows from \eqref{IDA} and \eqref{9}
\[
\begin{split}
 U_{X,\lda }(\xi)&= \left(\frac{\lda}{|\xi-X|}\right)^{n-2\sigma}U(X+\frac{\lda^2(\xi-X)}{|\xi-X|^2})\\
&\leq \left(\frac{\lda_0(x)}{|\xi-X|}\right)^{n-2\sigma}\sup\limits_{\B^{+}_{\mu}(X)}U\\
&\leq \left(\frac{\mu}{|\xi-X|}\right)^{n-2\sigma}\inf\limits_{\partial ''\B^{+}_{\mu}(X)}U\\
&\leq U(\xi).
\end{split}
\]

As the second step, we are going to obtain \eqref{Q}. For all $\lda\in (0,\lambda_0(x))$, $\xi\in\pa'' \B^+_{\mu}(X)$, we have $X+\frac{\lda^2(\xi-X)}{|\xi-X|^2}\in \B^+_{\mu}(X)$. It follows that
\begin{equation*}
\begin{split}
U_{X,\lda }(\xi)&=\left(\frac{\lambda}{|\xi-X|}\right)^{n-2\sigma}U\left(X+\frac{\lda^2(\xi-X)}{|\xi-X|^2}\right)
\leq\left(\frac{\lambda_0}{\mu}\right)^{n-2\sigma}\sup\limits_{\B_{\mu}^+(X)}U\\
&\leq \inf_{\partial ''{\B_{\mu}^+(X)}}U
\leq U(\xi).
\end{split}
\end{equation*}
The above inequality, together with
\[
U_{X,\lda }(\xi)=U(\xi)\quad\quad\mbox{on}\  \ \ \pa'' \B^+_{\lambda}(X),
\]
implies that for all $\lda \in(0,\lambda_0(x))$,
\be\label{W}
U_{X,\lda }(\xi)\leq U(\xi)\quad\quad \mbox{on }\ \ \ \partial ''\B^+_{\mu}(X)\cup \partial ''\B^+_{\lda}(X).
\ee

We will make use of the ¡°narrow domain technique¡± of Berestycki and Nirenberg from \cite{BN}, and show that,  for sufficiently small $\mu$, that for $\lda \in(0,\lambda_0(x))$
$$U_{X,\lda }(\xi)\leq U(\xi)\quad\quad\quad\mbox{in}\ \ \ \B^+_{\mu}(X)\backslash \B^+_{\lda}(X).$$
For simplicity, we denote $D:=\B_{\mu}^+(X)\backslash \B_{\lda}^+(X)$. A direct calculation gives that
\begin{equation}\label{4}
\begin{cases}
\mathrm{div}(t^{1-2\sigma}\nabla (U_{X,\lda }-U))=0& \text{in}\quad D,\\
\frac{\partial(U_{X,\lda }-U)}{\partial \nu^\sigma}
&\\=\alpha_1(u_{x,\lda }^{\frac{n+2\sigma}{n-2\sigma}}-u^{\frac{n+2\sigma}{n-2\sigma}})+\beta (u_{x,\lda }^{\frac{2\sigma}{n-2\sigma}}v_{x,\lda }^{\frac{n}{n-2\sigma}}- u^{\frac{2\sigma}{n-2\sigma}}v^{\frac{n}{n-2\sigma}})
\quad &\text{on}\quad \partial 'D.
\end{cases}
\end{equation}
Let $(U_{X,\lda }-U)^+:=\max(0,U_{X,\lda }-U)$ which equals to $0$ on $\pa''D$. Multiplying the first equation in \eqref{4} by $(U_{X,\lda }-U)^+$ and integrating by parts in $D$.  With the help of the Mean Value Theorem, we have
\begin{equation}\label{HH}
\begin{split}
&\int_{D} t^{1-2\sigma}|\nabla(U_{X,\lda }-U)^+|^2\\
=&\int_{D}\left[\alpha_1(u_{x,\lda }^{\frac{n+2\sigma}{n-2\sigma}}-u^{\frac{n+2\sigma}{n-2\sigma}})+\beta (u_{x,\lda }^{\frac{2\sigma}{n-2\sigma}}v_{x,\lda }^{\frac{n}{n-2\sigma}}- u^{\frac{2\sigma}{n-2\sigma}}v^{\frac{n}{n-2\sigma}})\right](u_{x,\lda }-u)^+\\
\leq&\int_{\partial 'D}\alpha_1\frac{n+2\sigma}{n-2\sigma}u_{x,\lda }^{\frac{4\sigma}{n-2\sigma}}\left((u_{x,\lda }-u)^{+}\right)^2\\
&+\int_{\partial 'D}\beta (u_{x,\lda }^{\frac{2\sigma}{n-2\sigma}}v_{x,\lda }^{\frac{n}{n-2\sigma}}- u^{\frac{2\sigma}{n-2\sigma}}v^{\frac{n}{n-2\sigma}})(u_{x,\lda }-u)^+\\
=:&I_1+I_2.
\end{split}
\end{equation}

First, we are going to estimate the term $I_1$. Using the H\"{o}lder inequality and the trace inequality, we have
\begin{equation}\label{I_1}
\begin{split}
I_1=&\alpha_1\frac{n+2\sigma}{n-2\sigma}\int_{\partial 'D}u_{x,\lda }^{\frac{4\sigma}{n-2\sigma}}\left((u_{x,\lda }-u)^{+}\right)^2\\
\leq&\alpha_1\frac{n+2\sigma}{n-2\sigma}\left(\int_{\partial 'D}u_{x,\lda }^{\frac{2n}{n-2\sigma}}\right)^{\frac{2\sigma}{n}}\left(\int_{\partial 'D}\left((u_{x,\lda }-u)^+\right)^{\frac{2n}{n-2\sigma}}\right)^{\frac{n-2\sigma}{n}}\\
\leq&C_1\left(\int_{B_{\mu}(x)}u^{\frac{2n}{n-2\sigma}}\right)^{\frac{2\sigma}{n}}\int_{D} t^{1-2\sigma}|\nabla(U_{X,\lda }-U)^+|^2.
\end{split}
\end{equation}
Here $C_1$ is a constant depending on $n$, $\sigma$, $\alpha_1$. Let us make an explanation for the last inequality.

Making a change of variables, for $\ y\in B_{\mu}(x)\backslash B_{\lda}(x)$,  define
\[
y_{\lambda}:=x+\frac{\lda^2(y-x)}{|y-x|^2}\in B_{\lda}(x)\subset B_{\mu}(x).
\]
Via a direct calculation, we have
\[
\frac{\partial y_{\lambda}}{\partial y}=\lda^2\left(\frac{1}{|y_{\lambda}-x|^2}I_n-2\frac{(y_{\lambda}-x)\otimes (y_{\lambda}-x)}{|y_{\lambda}-x|^4}\right),
\]
where $(y_{\lambda}-x)\otimes (y_{\lambda}-x)=((y_{\lambda}-x)_i(y_{\lambda}-x)_j)$. It follows that
\[
\det\left(\frac{\partial{y_\lambda}}{\partial y}\right)=\lda^{2n}|(y_{\lambda}-x)|^{-2n}.
\]
Therefore, we deduce  that
\begin{equation}\label{fact1}
\begin{split}
\int_{B_{\mu}(x)\backslash B_{\lda}(x)}u_{x,\lda }(y)^{\frac{2n}{n-2\sigma}}dy&\leq \int_{B_{\mu}(x)}\bigg(\frac{|y_{\lambda}-x| }{\lda}\bigg)^{2n}u(y_{\lambda})^{\frac{2n}{n-2\sigma}}\bigg(\frac{|y_{\lambda}-x| }{\lda}\bigg)^{-2n}dy_{\lambda}\\
&= \int_{B_{\mu}(x)}u(y_{\lambda})^{\frac{2n}{n-2\sigma}}dy_{\lambda}.
\end{split}
\end{equation}
Next, let us estimate the term $I_2$.  We claim that
\[
I_2\leq C_3\int_{\partial 'D}\left((u_{x,\lda }-u)^++(v_{x,\lda }-v)^+\right)(u_{x,\lda }-u)^+.
\]
Here $C_3$ is a constant depending on $n$, $\beta$,  $\sigma$, $x$. The detailed proof of this claim is provided  in Proposition \ref{QL}. Then through the H\"{o}lder inequality, it is not difficult to get
\begin{equation*}
\begin{split}
I_2
\leq&C_3|B_{\mu}(x)|^{\frac{2\sigma}{n}}\left(\int_{\partial 'D}\left((u_{x,\lda }-u)^+\right)^{\frac{2n}{n-2\sigma}}\right)^{\frac{n-2\sigma}{n}}\\
&+C_3|B_{\mu}(x)|^{\frac{2\sigma}{n}}\left(\int_{\partial 'D}\left((u_{x,\lda }-u)^+\right)^{\frac{2n}{n-2\sigma}}\right)^{\frac{n-2\sigma}{2n}}\left(\int_{\partial 'D}\left((v_{x,\lda }-v)^+\right)^{\frac{2n}{n-2\sigma}}\right)^{\frac{n-2\sigma}{2n}},
\end{split}
\end{equation*}
With the help of the trace inequality, we obtain
\begin{equation}\label{I2}
\begin{split}
I_2\leq&C\mu^{2\sigma}\int_{D} t^{1-2\sigma}|\nabla(U_{X,\lda }-U)^+|^2\\
&+C\mu^{2\sigma}\left(\int_{D} t^{1-2\sigma}|\nabla(U_{X,\lda }-U)^+|^2\right)^{\frac{1}{2}}\left(\int_{D} t^{1-2\sigma}|\nabla(V_{X,\lda }-V)^+|^2\right)^{\frac{1}{2}}.
\end{split}
\end{equation}
Here $C$ is a constant depending on $n$, $\beta$,  $\sigma$, $x$.

Define
\[
\widetilde{C_1}:=C_1\left(\int_{B_{\mu}(x)}u ^{\frac{2n}{n-2\sigma}}\right)^{\frac{2\sigma}{n}}+C\mu^{2\sigma}.
\]
From \eqref{HH}, \eqref{I_1} and \eqref{I2}, we have
\begin{equation*}
\begin{split}
(1-\widetilde{C_1})\|\nabla(U_{X,\lda }-U)^+\|_{L^2(t^{1-2\sigma},D)}
\leq C\mu^{2\sigma}\|\nabla(V_{X,\lda }-V)^+\|_{L^2(t^{1-2\sigma},D)},
\end{split}
\end{equation*}

By the fact that  $u\in C^{2}(\R^n)$, we can  choose $\mu$ sufficient small such that
\[
\frac{C\mu^{2\sigma}}{(1-\widetilde{C_1})}\leq\frac{1}{2}.
\]
It follows that
\begin{equation}\label{budengshi1}
\|\nabla(U_{X,\lda }-U)^+\|_{L^2(t^{1-2\sigma},D)}
\leq\frac{1}{2}\|\nabla(V_{X,\lda }-V)^+\|_{L^2(t^{1-2\sigma},D)},
\end{equation}
where $D:=\B_{\mu}^+(X)\backslash \B_{\lda}^+(X)$.
By the same argument, we can also obtain that
\begin{equation}\label{budengshi2}
\|\nabla(V_{X,\lda }-V)^+\|_{L^2(t^{1-2\sigma},D)}
\leq\frac{1}{2}\|\nabla(U_{X,\lda }-U)^+\|_{L^2(t^{1-2\sigma},D)}.
\end{equation}
Combining \eqref{budengshi1} with \eqref{budengshi2},~we have
$$\nabla(U_{X,\lda }(\xi)-U(\xi))^+=0\quad\quad \mbox{in  }\ \ D.$$
By \eqref{W}, we conclude that
\[
(U_{X,\lda }(\xi)- U(\xi))^+=0\quad\quad\mbox{on }\ \ \ D,
\]
which implies  \eqref{Q} and  Lemma \ref{lemma1} is proved.
\end{proof}
\begin{proof}[Proof of Lemma \ref{lemma2}]
From the definition of $\bar \lda(x)$, it is obviously that for all $\lambda\in(0,\bar\lda(x))$,
\be
U_{X,\lda}(\xi)\leq U(\xi),\ \ V_{X,\lda}(\xi)\leq V(\xi)\quad \mbox{in }\ \R^{n+1}_+\backslash \B^+_{\lda}(X).
\ee
For $\xi\in \R^{n+1}_+\backslash \B^+_{\lda}(X)$, let
$$\xi_{\lambda}:=X+\frac{\lda^2(\xi-X)}{|\xi-X|^2},$$
which implies that
\[
U(\xi_{\lambda})\leq\left(\frac{\lda }{|\xi_{\lambda}-X|}\right)^{n-2\sigma}U\left(X+\frac{\lda^2(\xi_{\lambda}-X)}{|\xi_{\lambda}-X|^2}\right)=U_{X,\lda}(\xi_{\lambda})\quad \mbox{in }\ \overline{\B^+_{\lda}(X)}\backslash \{X\}.
\]
Then we can say
\be\label{M}
U(\xi)\leq U_{X,\lda}(\xi)\quad \mbox{in }\ \B^+_{\lda}(X)\backslash \{X\}.
\ee
We prove Lemma \ref{lemma2} by contradiction. Without loss of generality, we suppose $U_{X,\bar\lda(x)}\not\equiv U$. Next, we will show that there exists a positive constant $\va$ such that for all $\lda\in(\bar\lda(x),\bar\lda(x)+\va)$,
\be\label{Zhong}
U(\xi)\leq U_{X,\lda}(\xi)\quad \mbox{in }\ \B^+_{\lda}(X)\backslash \{X\},
\ee
which contradicts with the definition of $\bar\lda(x)$.  For this purpose, we first claim that  if $U_{X,\bar\lda(x)}\not\equiv U$, it follows that $V_{X,\bar\lda(x)}\not\equiv V$.

In fact, if $V_{X,\bar\lda(x)}\equiv V$, by a direct calculation gives that
\[
0=\frac{\partial(V_{X,\bar\lda(x) }-V)}{\partial \nu^\sigma}(x,0)=\beta (u_{x,\bar\lda(x) }^{\frac{2\sigma}{n-2\sigma}}-u^{\frac{2\sigma}{n-2\sigma}})v^{\frac{n}{n-2\sigma}}\neq0,
\]
which is a contradiction.

Now, let us divide the region  $\B^+_{\lda}(X)\backslash \{X\}$ into three parts,
\[
\begin{array}{ll}
K_1:=\left\{\xi\in \B^+_{\lda}(X)\ \big| \ 0<|\xi-X|< \delta_1\right\},\vspace{0.2cm}\\
K_2:=\left\{\xi\in \B^+_{\lda}(X)\ \big| \ \delta_1\leq|\xi-X|\leq\bar\lda-\delta_2\right\},\vspace{0.2cm}\\
K_3:=\left\{\xi\in \B^+_{\lda}(X)\ \big| \ \bar\lda-\delta_2\leq|\xi-X|\leq\lda\right\},
\end{array}
\]
where $\delta_1$,$\delta_2$ will be fixed later. To obtain \eqref{Zhong} it  suffices to prove that it holds respectively on $K_1$, $K_2$, $K_3$.

Combining
\begin{equation*}
\begin{cases}
\mathrm{div}(t^{1-2\sigma}\nabla (U_{X,\bar\lda(x) }-U)=0& \text{in} \ \B_{\bar\lda(x)}^+(X),\\
\frac{\partial (U_{X,\bar\lda(x) }-U)}{\partial \nu^\sigma} \geq0 &\text{on}\ \partial'\B_{\bar\lda(x)}^+(X)\setminus \{X\},
\end{cases}
\end{equation*}
with the fact that $U_{X,\bar\lda(x)}\not\equiv U$, in view of  the Harnack inequality \cite[Proposition 2.6]{JLX}, which implies that
\[
U_{X,\bar\lda(x)}(\xi)-U(\xi)>0\quad \mbox{in } \ \B_{\bar\lda(x)}^+(X)\backslash\{X\}.
\]
By Proposition \ref{prop:liminf}, we have
\[
\liminf_{\xi\to X} (U_{X,\bar\lda}(\xi)-U(\xi))>0.
\]
As a result, there exist two positive constants $\delta_{1}$ and $C_1$ such that
\be\label{O1}
U_{X,\bar\lda(x)}(\xi)-U(\xi)>C_1\quad\text{ in } \ K_{1}.
\ee
Choose $\varepsilon_{1}< \delta_{1}$ small such that for all $\lda\in(\bar\lda(x),\bar\lda(x)+\varepsilon_{1})$,
\be\label{O2}
U_{X,\bar\lda(x)}\left(X+\frac{\overline{\lambda}^2(x)}{\lambda^2}(\xi-X)\right)-U_{X,\bar\lda(x)}(\xi)>-C_1/2\quad\text{ in } \ K_{1},
\ee
and
\be\label{O3}
\left(\frac{\bar\lda(x)}{\bar\lda(x)+\varepsilon_{1}}\right)^{n-2\sigma}\left(U(\xi)+C_1/2\right)\geq U(\xi)+C_1/4.
\ee
By a directly calculations, we have
\[
U_{X,\lda }(\xi)= \left(\frac{\bar\lda(x)}{\lda}\right)^{n-2\sigma}U_{X,\bar\lda(x)}\left(X+\frac{\overline{\lambda}^2(x)}{\lambda^2}(\xi-X)\right)
\]
It follows that
\begin{equation*}
\begin{split}
U_{X,\lda }(\xi)&\geq\left(\frac{\bar\lda(x)}{\bar\lda(x)+\varepsilon_{1}}\right)^{n-2\sigma}U_{X,\bar\lda(x)}\left(X+\frac{\overline{\lambda}^2(x)}{\lambda^2}(\xi-X)\right)\\
&\geq \left(\frac{\bar\lda(x)}{\bar\lda(x)+\varepsilon_{1}}\right)^{n-2\sigma}\left(U\left(\xi\right)+C_1/2\right)\\
&\geq U(\xi)+C_1/4,
\end{split}
\end{equation*}
where \eqref{O1}, \eqref{O2} is used in the second inequality and \eqref{O3} is used in the last inequality.
Consequently, for any $ \lda \in(\bar \lda(x),\bar \lda(x)+\varepsilon_{1})$,
\be\label{H112}
U_{X,\lda}(\xi)\geq U(\xi)\ \quad\text{ in }\  K_{1}.
\ee

For $\delta_2$ small, which will be fixed later. Since $K_{2}$ is compact, there exists a positive constant $C_2$ such that
\[
U_{X,\bar\lda(x)}(\xi)-U(\xi)>C_2 \ \quad\text{ in } \  K_{2}.
\]
By the uniform continuity of $U$  on compact sets, there exists a positive constant $\va_{2}>0$ small such that for all $ \lambda\in(\bar \lda(x),\bar \lda(x)+\varepsilon_{2})$,
\[
U_{X,\lda}(\xi)-U_{X,\bar\lda(x)}(\xi)>-C_2/2\ \quad\text{ in }  \ K_{2}.
\]
Hence, for all $ \lambda\in(\bar \lda(x),\bar \lda(x)+\varepsilon_{2})$,
\be\label{H2}
U_{X,\lda}(\xi)-U(\xi)>C_2/2\ \quad\text{ in } \ K_{2}.
\ee

Now let us focus on the region $K_3$.
Using the narrow domain technique as that in Lemma \ref{lemma1}, we can choose $\delta_2$
small ( notice that we can choose $\va$ as small as we want less then $\varepsilon_1$ and $\va_2$ such that for $ \lda \in(\bar \lda(x),\bar \lda(x)+\va)$,
\be\label{H3b}
U(\xi)\leq U_{X,\lda}(\xi),\ \ V(\xi)\leq V_{X,\lda}(\xi) \ \quad\text{ in } \ K_3.
\ee
Together with \eqref{H112}, \eqref{H2} and \eqref{H3b}, we can see that the moving sphere procedure may continue beyond $\bar \lda(x)$ where we reach a contradiction. We can see that the moving sphere procedure may continue beyond $\bar \lda(x)$ for $U$ and $V$ where we reach a contradiction. Then we complete the proof of Lemma 2.
\end{proof}
\begin{proof}[Proof of Lemma \ref{lemma3}]
For some $x$, if $\bar \lda(x)=+\infty$, then for all $\lda>0$,
\[
U_{X, \lda}(\xi)\leq U(\xi)\quad \mbox{in }\  \R_+^{n+1}\backslash \B_{\lda}^+(X),
\]
it follows that
\[
\limsup_{\xi\rightarrow+\infty}|\xi|^{n-2\sigma}U_{X, \lda}(\xi)\leq \limsup_{\xi\rightarrow+\infty}|\xi|^{n-2\sigma}U(\xi).
\]
It is not hard to see that,
\[
\begin{split}
\limsup_{\xi\rightarrow+\infty}|\xi|^{n-2\sigma}U_{X, \lda}(\xi)=&\limsup_{\xi\rightarrow+\infty}\frac{|\xi|^{n-2\sigma}}{|\xi-X|^{n-2\sigma}}\lambda^{n-2\sigma}U\left( X+\frac{\lambda^2(\xi-X)}{|\xi-X|^2}\right)
=\lambda^{n-2\sigma}U(X),
\end{split}
\]
which implies that  for any $\lda>0$,
\[
\lambda^{n-2\sigma}U(X)\leq\limsup_{\xi\rightarrow+\infty}|\xi|^{n-2\sigma}U(\xi).
\]
Therefore,
\[
\limsup_{\xi\rightarrow+\infty}|\xi|^{n-2\sigma}U(\xi)=+\infty.
\]
On the other hand, if for some $\widehat{x}$, $\bar \lda(\widehat{x})<+\infty$, then by Lemma \ref{lemma2},
$U_{X,\bar\lda(\widehat{x})}(\xi)=U(\xi)$, it follows that
\[
\limsup_{\xi\rightarrow+\infty}|\xi|^{n-2\sigma}U(\xi)=\limsup_{\xi\rightarrow+\infty}|\xi|^{n-2\sigma}U_{\widehat{X}, \bar\lda(\widehat{x})}(\xi)= \bar\lda^{n-2\sigma}(\widehat{x})U(\widehat{X})<+\infty.
\]
It is a contradiction.
\end{proof}
\begin{proof}[Proof of Lemma \ref{lemma4}]
By Proposition \ref{CHANG2}, we can prove this lemma.

\end{proof}
\section{Global Solutions with an Isolated Singularity}\label{F3}
\subsection{Proof of Theorem \ref{ZHUYAO}}
Since the method for $u$ and $v$ are same, we just give a proof for $u$. To prove that $u$ is radial and monotonically decreasing radially it suffices to show that $u$ is symmetrical about any  hyperplane which through the origin and it is monotone decreasing along the normal direction. Without loss of generality, let us prove that $u$ is symmetric about the hyperplane $\{y_1=0\}$ and it is monotone decreasing along the $y_1$ axis.

For all $x\in\R^n\setminus\{0\}$, assume that for any $\lambda \in (0,|x|)$,
\be\label{eq:small stop1}
u_{x,\lda}(y)\le u(y),\quad v_{x,\lda}(y)\le v(y)\quad\quad\mbox{in }\ \ \R^n\backslash (B_\lda(x)\cup\{0\}).
\ee
Let $t$, $s\in \R$  satisfy $t\leq s$, $t+s>0$ and $m>\max \{s,\frac{st}{s+t}\}$, then $0<(m-s)(m-t)<m^2$. With the help of \eqref{eq:small stop1}, choosing $y=te_1$, $x=me_1$ and  $\lda^2=(m-s)(m-t)$, we have
\[
\left(\frac{\sqrt{(m-s)(m-t)}}{m-t}\right)^{n-2\sigma}u\left[\left(m+\frac{(m-s)(m-t)}{t-m}\right)e_1\right]\leq u(te_1),
\]
where the unite vector $e_1=(1,0,\cdots,0)\in\R^n$. That is,
\[
\left(\frac{m-s}{m-t}\right)^{\frac{n-2\sigma}{2}}u\left(se_1\right)\leq u(te_1).
\]
After sending $m\rightarrow \infty$, it follows that
\be\label{R}
u(se_1)\leq u(te_1).
\ee
For $s>0$, let $t\rightarrow -s$, we obtain that
\be\label{DAN}
u(se_1)\leq u(-se_1).
\ee
By the same argument, choosing $y=-te_1$, $x=-me_1$ and  $\lda^2=(m-s)(m-t)$, it follows that
\[
u(-se_1)\leq u(-te_1).
\]
For $s>0$, let $t\rightarrow -s$, we have
\be\label{DAN1}
u(-se_1)\leq u(se_1).
\ee
Combining \eqref{DAN} with \eqref{DAN1}, we deduce that $u$ is symmetric about the hyperplane $\{y_1=0\}$.

On the other hand, for $0<t<s$, a consequence of  \eqref{R} is that $u$ is monotone decreasing along the axis of $y_1$.

Therefore, to finish the proof of Theorem \ref{ZHUYAO}, we only need to prove \eqref{eq:small stop1}.

\subsection{Proof of \eqref{eq:small stop1}}
To obtain  \eqref{eq:small stop1}, it suffices to prove that for any $x\in\R^n\setminus\{0\}$, $X=(x,0)$, $\lambda >0$, and for any $\lda\in(0,|x|)$
\be\label{33}
U_{X,\lda}(\xi)\le U(\xi),\ \ V_{X,\lda}(\xi)\le V(\xi)\quad\quad \mbox{in}\ \ \R^{n+1}_{+}\backslash \B^+_\lda(X),
\ee
where  $U$, $V$ defined as \eqref{DD}, and $U_{X,\lda }(\xi)$, $V_{X,\lda }(\xi)$ denote the Kelvin transformation of $U$, $V$.

For the sake of \eqref{33}, let us define
\[
\bar \lda(x):=\sup \big\{\lda(x)\in (0, |x|)\ \big|\ U_{X,\lda}(\xi)\leq U(\xi),\ V_{X,\lda}(\xi)\le V(\xi) \ \mbox{in} \ \R^{n+1}_{+}\backslash \B^+_\lda(X),~\forall~ \lambda\in(0, \lambda(x))\big\}.
\]
By the above argument, it is clear to see that we only need to obtain
\[
\bar \lda(x)=|x|.
\]
For this purpose, we need  the following lemma to guarantee that the set over which we are taking the supremum is non-empty and then $\bar\lambda(x)$ is well defined.
\begin{lem}\label{lemma4.1} There exists $\lda_0(x)\in (0,|x|)$ such that  for all $\lda\in(0,\lda_0(x))$,
\be\label{eq:bigger}
U_{X,\lda}(\xi)\le U(\xi),\ \ V_{X,\lda}(\xi)\le V(\xi)\quad\quad \mbox{\rm in}\ \ \R^{n+1}_{+}\backslash \B^+_\lda(X).
\ee
\end{lem}
To get \eqref{33},  we shall prove
\begin{lem}\label{lemma4.2}
\be\label{Weq:gseq}
\bar \lda(x)=|x|.
\ee
\end{lem}

Without loss of generality, we assume that
\[
\limsup_{x\rightarrow 0}u=\infty,
\]
namely $0$ is a non-removable singularity of $U$. Moreover, by the previous argument, we know that
$U$, $V\in C^2(\R^{n+1}_{+})\cap C(\overline{\R^{n+1}_{+}}\setminus\{0\})$, and
\be\label{2}
\begin{cases}
\begin{aligned}
&\mathrm{div}(t^{1-2\sigma} \nabla U)=0 &\quad&\mbox{in }\ \ \R^{n+1}_{+},\\
&\mathrm{div}(t^{1-2\sigma} \nabla V)=0 &\quad&\mbox{in }\ \ \R^{n+1}_{+},\\
&\frac{\partial U}{\partial \nu^\sigma}=\alpha_1u^{\frac{n+2\sigma}{n-2\sigma}}+\beta u^{\frac{2\sigma}{n-2\sigma}}v^{\frac{n}{n-2\sigma}} &\quad&\mbox{on }\ \ \partial'\R^{n+1}_{+}\backslash\{0\},\\
&\frac{\partial V}{\partial \nu^\sigma}=\alpha_2v^{\frac{n+2\sigma}{n-2\sigma}}+\beta v^{\frac{2\sigma}{n-2\sigma}}u^{\frac{n}{n-2\sigma}} &\quad&\mbox{on }\ \ \partial'\R^{n+1}_{+}\backslash\{0\}.
\end{aligned}
\end{cases}
\ee
\begin{proof}[Proof of Lemma \ref{lemma4.1}] We are going to show that there exist $\mu$ and $\lambda_0(x)$ satisfying $0<\lambda_0(x)<\mu$, ~which may depend on ~$x$,~ such that for all $\lda \in (0,\lambda_0(x))$,
\be\label{QQ}
U_{X,\lda }(\xi)\leq U(\xi),\ \ V_{X,\lda }(\xi)\leq V(\xi)\quad\quad \mbox{in}\ \ \overline{\B^+_\mu(X)}\backslash\B^+_\lda(X).
\ee
Then we will prove that for all $\lda \in (0,\lda_0(x))$,
\be\label{NQ}
U_{X,\lambda}(\xi)\leq U(\xi),\ \ V_{X,\lda }(\xi)\leq V(\xi)\quad\quad \mbox{in}\ \ \R^{n+1}_+\backslash\overline{\B^+_\mu(X)}.
\ee
The proof of \eqref{QQ} follows exactly the same as that for the the proof of the Lemma \ref{lemma1}.
Then, we just need to prove  \eqref{NQ}. By Proposition \ref{PP},  we have
\be\label{Weq:cl22U}
U(\xi)\geq \left(\frac{\mu}{|\xi-X|}\right)^{n-2\sigma}\inf_{\partial'' \B^{+}_{\mu}(X)} U\quad\quad \mbox{in}\ \  \R^{n+1}_+\backslash \overline{\B^{+}_{\mu}(X)}.
\ee
Then for all $\xi\in \R^{n+1}_+\backslash \B^{+}_{\mu}(X)$  and $\lda\in(0,\lda_0(x))\subset (0,\mu)$, we obtain that
\[
\begin{split}
 U_{X,\lda }(\xi)&= \left(\frac{\lda}{|\xi-X|}\right)^{n-2\sigma}U(X+\frac{\lda^2(\xi-X)}{|\xi-X|^2})\leq
 \left(\frac{\lda_0(x)}{|\xi-X|}\right)^{n-2\sigma}\sup\limits_{\B^+_{\mu}(X)}U\\
&\leq \left(\frac{\mu}{|\xi-X|}\right)^{n-2\sigma}\inf\limits_{\partial ''\B^+_{\mu}(X)}U
\leq U(\xi),
\end{split}
\]
where \eqref{Weq:cl22U} is used in the last inequality. Lemma \ref{lemma4.1} is proved.
\end{proof}
\begin{proof}[Proof of Lemma \ref{lemma4.2}]
By Lemma \ref{lemma4.1}, $\bar \lda(x)$ is well defined, and we also know that for $x\neq0$, $\bar \lda(x)\leq|x|$.   From the definition of $\bar \lda(x)$, it is obvious to see that for any $\lambda\in(0,\bar\lda(x)]$,
\be\label{M1}
U_{X,\lda}(\xi)\leq U(\xi),\quad V_{X,\lda}(\xi)\leq V(\xi)\quad \mbox{in }\ \R^{n+1}_+\backslash \B^+_{\lda}(X).
\ee
We prove Lemma \ref{lemma4.2} by contradiction. That is, suppose $\bar \lda(x)<|x|$ for some $x\neq 0$. We want to show that there exists $\va\in (0,\frac{|x|-\bar\lda(x)}{2})$ such that for any $\lda\in(\bar\lda(x),\bar\lda(x)+\va)$,
\be\label{Fenjie}
U_{X,\lda}(\xi)\leq U(\xi),\quad V_{X,\lda}(\xi)\leq V(\xi)\quad \mbox{in }\ \R^{n+1}_+\backslash \B^+_{\lda}(X),
\ee
which contradicts with the definition of $\bar\lda(x)$, then we obtain $\bar\lda(x)=|x|$.

We divide the region into two parts,
\[
\begin{array}{ll}
K_1:=\left\{\xi\in\R_{+}^{n+1}\ \big|  \ |\xi-X|\geq \bar\lda(x)+\delta_2\right\},\vspace{0.2cm}\\
K_2:=\left\{\xi\in\R^{n+1}_+\ \big| \ \lda\leq|\xi-X|\leq \bar\lda(x)+\delta_2\right\},
\end{array}
\]
where $\delta_2$ will be fixed later. Then in order to obtain \eqref{Fenjie} it  suffices to prove that it holds respectively on $K_1$, $K_2$.

From \eqref{M1}, we have obtained
\be
U_{X,\bar\lda(x)}(\xi)\leq U(\xi),\quad V_{X,\bar\lda(x)}(\xi)\leq V(\xi)\quad \mbox{in }\ \R^{n+1}_+\backslash \B^+_{\bar\lda(x)}(X).
\ee
Besides, by the fact that
\begin{equation*}
\begin{split}
\lim_{\xi\to 0}U_{X,\bar\lda(x) }(\xi)&= \lim_{\xi\to 0}\left(\frac{\bar\lda(x)}{|\xi-X|}\right)^{n-2\sigma}U\left(X+\frac{\bar\lda(x)^2(\xi-X)}{|\xi-X|^2}\right)\\
&=\left(\frac{\bar\lda(x)}{|X|}\right)^{n-2\sigma}U\left(X-\frac{\bar\lda(x)^2X}{|X|^2}\right)< \infty,
\end{split}
\end{equation*}
and the strong maximum principle, we have
\be\label{BO}
U_{X,\bar\lda(x)}(\xi)< U(\xi),\quad V_{X,\bar\lda(x)}(\xi)< V(\xi)\quad \mbox{in }\ \R^{n+1}_+\backslash \overline{\B^+_{\bar\lda(x)}(X)}.
\ee
Via a  calculation, it follows that
\begin{equation}
\begin{cases}
\mathrm{div}(t^{1-2\sigma}\nabla (U-U_{X,\bar\lda(x) }))=0&\text{in }\ K_1,\\
\frac{\partial (U-U_{X,\bar\lda(x) })}{\partial \nu^\sigma}
&\\=\alpha_1(u^{\frac{n+2\sigma}{n-2\sigma}}-u^{\frac{n+2\sigma}{n-2\sigma}}_{x,\bar\lda(x))}+\beta (u^{\frac{2\sigma}{n-2\sigma}}v^{\frac{n}{n-2\sigma}}-u_{x,\bar\lda(x) }^{\frac{2\sigma}{n-2\sigma}}v_{x,\bar\lda(x)}^{\frac{n}{n-2\sigma}})\quad &\text{on }\ \partial'K_1.
\end{cases}
\end{equation}
As a result,
\begin{equation}
\begin{cases}
\mathrm{div}(t^{1-2\sigma}\nabla (U-U_{X,\bar\lambda(x) }))=0& \text{in }\ K_1,\\
\frac{\partial (U-U_{X,\bar\lambda(x) })}{\partial \nu^\sigma}\geq0\quad &\text{on }\ \partial'K_1.
\end{cases}
\end{equation}
Through a combination of Proposition \ref{PP} and \eqref{BO}, it follows that
\be\label{S}
(U-U_{X,\bar\lambda(x) })(\xi)\geq \left(\frac{\bar\lambda(x)+\delta_2}{|\xi-X|}\right)^{n-2\sigma}\inf_{\partial'' \B^{+}_{\bar\lambda(x)+\delta_2}(X)} (U-U_{X,\bar\lambda(x) })\ \quad\text{ in } \ K_1,
\ee
and
\[
\inf_{\partial'' \B^{+}_{\bar\lambda(x)+\delta_2}(X)} (U-U_{X,\bar\lambda(x) })>0,
\]
By the uniform continuity of $U$ on compact sets, there exists $\varepsilon_{1}<\frac{|x|-\bar\lambda(x)}{2}$ small such that for any $\lambda\in(\bar \lambda(x), \bar \lambda(x)+\varepsilon_{1}$),
\be\label{WE}
|U_{X,\bar\lambda(x)}(\xi)-U_{X,\lambda}(\xi)|\leq\frac{1}{2}\left(\frac{\bar\lambda(x)+\delta_2}{|\xi-X|}\right)^{n-2\sigma}\inf_{\partial'' \B^{+}_{\bar\lambda(x)+\delta_2}(X)} (U-U_{X,\bar\lambda(x) })\ \quad\text{ in } \ K_1.
\ee
Indeed, notice that $\xi_{\bar\lambda(x)}:=X+\frac{\bar\lambda^2(x)(\xi-X)}{|\xi-X|^2}$,  $\xi_{\lambda}:=X+\frac{\lambda^2(\xi-X)}{|\xi-X|^2}\in \overline{\B^+_{\frac{|x|+\bar\lambda(x)}{2}}(X)} $ and
\begin{equation*}
\begin{split}
|U_{X,\bar\lambda(x)}(\xi)-U_{X,\lambda}(\xi)|=&\left|\left(\frac{\bar\lambda(x)}{|\xi-X|}\right)^{n-2\sigma}U(\xi_{\bar\lambda(x)})-\left(\frac{\lambda}{|\xi-X|}\right)^{n-2\sigma}U\left(\xi_{\lambda}\right)\right|\\
\leq&\left(\frac{\bar\lambda(x)}{|\xi-X|}\right)^{n-2\sigma}\left|U(\xi_{\bar\lambda(x)})-U\left(\xi_{\lambda}\right)\right|
+\frac{|\bar\lambda^{n-2\sigma}(x)-\lambda^{n-2\sigma}|}{|\xi-X|^{n-2\sigma}}U\left(\xi_{\lambda}\right)\\
=&\left(\frac{\bar\lambda(x)+\delta_2}{|\xi-X|}\right)^{n-2\sigma}\left(\frac{\bar\lambda(x)}{\bar\lambda(x)+\delta_2}\right)^{n-2\sigma}\left|U(\xi_{\bar\lambda(x)})-U\left(\xi_{\lambda}\right)\right|\\
&+\left(\frac{\bar\lambda(x)+\delta_2}{|\xi-X|}\right)^{n-2\sigma}\frac{|\bar\lambda^{n-2\sigma}(x)-\lambda^{n-2\sigma}|}{|\bar\lambda(x)+\delta_2|^{n-2\sigma}}U\left(\xi_{\lambda}\right).
\end{split}
\end{equation*}
By the uniform continuity of $U$ on compact sets, we can choose $\varepsilon_{1}$ sufficient small, such that for all $\lambda\in(\bar \lambda(x), \bar \lambda(x)+\varepsilon_{1}$),
\[
\left(\frac{\bar\lambda(x)}{\bar\lambda(x)+\delta_2}\right)^{n-2\sigma}\left|U(\xi_{\bar\lambda(x)})-U\left(\xi_{\lambda}\right)\right|
\leq \frac{1}{4}\inf_{\partial'' \B^{+}_{\bar\lambda(x)+\delta_2}(X)} (U-U_{X,\bar\lambda(x) }),
\]
and
\[
\frac{|\bar\lambda^{n-2\sigma}(x)-\lambda^{n-2\sigma}|}{|\bar\lambda(x)+\delta_2|^{n-2\sigma}}U\left(\xi_{\lambda}\right)
\leq \frac{1}{4}\inf_{\partial'' \B^{+}_{\bar\lambda(x)+\delta_2}(X)} (U-U_{X,\bar\lambda(x) }).
\]
Then \eqref{WE} follows.

Combining \eqref{S} with \eqref{WE}, we conclude that for any $ \lambda \in(\bar \lambda(x),\bar \lambda(x)+\varepsilon_{1})$,
\[
U_{X,\lambda}(\xi)\leq U(\xi) \ \quad\text{ in } \ K_1.
\]
Hence, we have
\[
U_{X,\lambda}(\xi)\leq U(\xi),\quad V_{X,\lambda}(\xi)\leq V(\xi) \ \quad\text{ on } \ \pa'K_2.
\]
Using the narrow domain technique as that the proof of \eqref{Q}, we can choose $\delta_2$
small (notice that we can choose $\va$ as small as we want less then $\varepsilon_1$ such that for $ \lambda \in(\bar \lambda(x),\bar \lambda(x)+\va)$, we have
\be\label{H3}
U_{X,\lambda}(\xi)\leq U(\xi),\quad V_{X,\lambda}(\xi)\leq V(\xi) \ \quad\text{ in } \ K_2.
\ee
From the above argument, we can see that the moving sphere procedure may continue beyond $\bar \lambda(x)$ where we reach a contradiction.
\end{proof}

\section{Proof of Theorem \ref{thm:a}}\label{F2}
In order to finish the proof of Theorem \ref{thm:a}, we divide the proof into two parts. As the first part, we will obtain the upper bound near an isolated singularity by the blow up analysis and the method of moving sphere in Section \ref{Upper}. One consequence of this upper bound is a Harnack inequality (Lemma \ref{lem:spherical harnack}), which will be used frequently in the second part. In the second part, we use the Pohozaev integral to prove Theorem \ref{14}, Theorem \ref{15}. A combination of Theorem \ref{14} and Theorem \ref{15}, we can see that if $\liminf_{|x|\to 0}|x|^{\frac{n-2\sigma}{2}}(u+v)(x)=0$, then $(u,v)$ can be extended as a continuous function at the origin $0$. Otherwise, we obtain the lower bound near an isolated singularity.
\subsection{A Upper Bound near an Isolated Singularity}\label{Upper}
\begin{thm}\label{thm2}
Let $(u,v)\in C^{2}(B_1\backslash \{0\})\cap L_{\sigma}(\R^n)$ be a solution of \eqref{91}, then
\[
\limsup_{|x|\to 0} |x|^{\frac{n-2\sigma}{2}}u(x)<\infty,\ \
\limsup_{|x|\to 0} |x|^{\frac{n-2\sigma}{2}}v(x)<\infty.
\]
\end{thm}
\begin{proof} Without loss of generality, we assume that $u$, $v$ are continuous to the boundary $\partial B_1$. Suppose the contrary that there exists a sequence $\{x_j\} \subset B_1$ such that
\[
x_j\to 0\quad \mbox{as } j\to \infty,
\]
and
\be\label{eq:cl1}
|x_j|^{\frac{n-2\sigma}{2}}(u+v)(x_j)\to \infty\quad \mbox{as }j\to \infty.
\ee
Define
\[
w(x):=(u+v)(x).
\]
Consider
\[
h_j(x):=\left(\frac{|x_j|}{2}-|x-x_j|\right)^{\frac{n-2\sigma}{2}} w(x)\quad\quad\mbox{in }\ \ B_{|x_j|/2}(x_j).
\]
Let $|\bar x_j-x_j|<\frac{|x_j|}{2}$ satisfy
\[
h_j(\bar x_j)=\max_{|x-x_j|\leq \frac{|x_j|}{2}}h_j(x),
\]
and
\[
2\mu_j:=\frac{|x_j|}{2}-|\bar x_j-x_j|.
\]
Then
\be \label{eq:cl2}
0<2\mu_j\leq \frac{|x_j|}{2}\quad\mbox{and}\quad \frac{|x_j|}{2}-|x-x_j|\ge\mu_j\quad\quad\mbox{in }\ \ B_{\mu_j}(\overline{x}_j).
\ee
By the definition of $h_j$, we have
\be \label{eq:cl3}
(2\mu_j)^{\frac{n-2\sigma}{2}}w(\bar x_j)=h_j(\bar x_j)\ge h_j(x)\ge (\mu_j)^{\frac{n-2\sigma}{2}}w(x)\quad\quad\mbox{in }\ \ B_{\mu_j}(\overline{x}_j).
\ee
Therefore,
\be\label{youjie}
\frac{w(x)}{w(\bar x_j)}\leq2^{\frac{n-2\sigma}{2}}\quad\quad\mbox{in }\ \ B_{\mu_j}(\bar x_j).
\ee
On the other hand,
\be\label{eq:cl4}
(2\mu_j)^{\frac{n-2\sigma}{2}}w(\bar x_j)=h_j(\bar x_j)\ge h_j(x_j)= \left(\frac{|x_j|}{2}\right)^{\frac{n-2\sigma}{2}}w(x_j)\to \infty\quad\mbox{as }j\to\infty.
\ee
Now, define
\[
U_j(y,t):=\frac{1}{w(\bar x_j)}U\left(\bar x_j+\frac{y}{w(\bar x_j)^{\frac{2}{n-2\sigma}}},\frac{t}{w(\bar x_j)^{\frac{2}{n-2\sigma}}}\right ) \quad \mbox{in }\  \om_j,
\]
\[
V_j(y,t):=\frac{1}{w(\bar x_j)}V\left(\bar x_j+\frac{y}{w(\bar x_j)^{\frac{2}{n-2\sigma}}},\frac{t}{w(\bar x_j)^{\frac{2}{n-2\sigma}}}\right ) \quad \mbox{in }\  \om_j,
\]
where $U$ and $V$ are defined as \eqref{DD}, and
\[
\overline{\om_j}:=\left\{(y,t)\in \overline{\R_+^{n+1}}| \left(\bar x_j+\frac{y}{w(\bar x_j)^{\frac{2}{n-2\sigma}}},\frac{t}{w(\bar x_j)^{\frac{2}{n-2\sigma}}}\right)\in  \overline{\B^+_{1}}\setminus \{0\} \right\}.
\]
Let
\[
u_j(y):=U_j(y,0),\quad\quad v_j(y):=V_j(y,0),
\]
It follows that
\be\label{11}
\begin{cases}
\begin{aligned}
&(-\Delta)^{\sigma}u_j=\alpha_1u_j^{\frac{n+2\sigma}{n-2\sigma}}+\beta u_j^{\frac{2\sigma}{n-2\sigma}}v_j^{\frac{n}{n-2\sigma}} &\quad&\mbox{in }\ \   \partial'\om_j,\\
&(-\Delta)^{\sigma}v_j=\alpha_2v_j^{\frac{n+2\sigma}{n-2\sigma}}+\beta v_j^{\frac{2\sigma}{n-2\sigma}}u_j^{\frac{n}{n-2\sigma}} &\quad&\mbox{in }\ \   \partial'\om_j.
\end{aligned}
\end{cases}
\ee

Moreover, $u_j(0)+ v_j(0)=1$. From \eqref{youjie}, we obtain for all $|\bar x_j+\frac{y}{w(\bar x_j)^{\frac{2}{n-2\sigma}}}-\bar x_j|\leq \mu_j$,
\begin{equation*}
\begin{split}
u_j(y)&=\frac{1}{w(\bar x_j)}u\left(\bar x_j+\frac{y}{w(\bar x_j)^{\frac{2}{n-2\sigma}}}\right )
\leq \frac{1}{w(\bar x_j)}(u+v)\left(\bar x_j+\frac{y}{w(\bar x_j)^{\frac{2}{n-2\sigma}}}\right )\\
&=\frac{1}{w(\bar x_j)}w\left(\bar x_j+\frac{y}{w(\bar x_j)^{\frac{2}{n-2\sigma}}}\right )\leq2^{\frac{n-2\sigma}{2}},
\end{split}
\end{equation*}
Therefore, we conclude that
\[
u_j(y),\ \ v_j(y)\leq2^{\frac{n-2\sigma}{2}}\quad \mbox{in}  \ \  B_{R_j},
\]
where $$R_j=\mu_jw(\bar x_j)^{\frac{2}{n-2\sigma}}\rightarrow \infty\quad\quad\mbox{as }\ \ j\rightarrow \infty.$$
It follows that
\[
u_j^{\frac{2\sigma}{n-2\sigma}}v_j^{\frac{n}{n-2\sigma}},\ \ v_j^{\frac{2\sigma}{n-2\sigma}}u_j^{\frac{n}{n-2\sigma}}\leq2^{\frac{n+2\sigma}{2}}\quad \mbox{in}\  \ B_{R_j}.
\]
By Proposition \ref{prop:guji1}, there exists $\gamma \in (0,1)$, such that $u_j$, $v_j\in C^{\gamma}(B_{R_j/2})$. Bootstrapping use Proposition \ref{prop:guji}, there exists $\al$ and for every $R>1$, such that $u_j$, $v_j\in C^{2,\al}(B_{R})$. Moreover,
\[
\|u_j\|_{C^{2,\al}(B_{R})}, \ \|v_j\|_{C^{2,\al}(B_{R})}\leq C(R),
\]
where $C$ is  independent on $j$. Thus, after passing to a subsequence,  there exist $\widetilde{u}$, $\widetilde{v}\in C^2(\R^n)$ such that
\[
\begin{cases}
u_j&\rightarrow \widetilde{u}\quad\mbox{in }\ C^2_{{\rm{loc}}}(\R^n),\\
v_j&\rightarrow \widetilde{v}\quad\mbox{in }\ C^2_{{\rm loc}}(\R^n),
\end{cases}
\]
and
\[
\begin{cases}
\begin{aligned}
&(-\Delta)^{\sigma}\widetilde{u}=\alpha_1\widetilde{u}^{\frac{n+2\sigma}{n-2\sigma}}+\beta \widetilde{u}^{\frac{2\sigma}{n-2\sigma}}\widetilde{v}^{\frac{n}{n-2\sigma}} &\quad&\mbox{in }\ \ \R^{n},\\
&(-\Delta)^{\sigma}\widetilde{v}=\alpha_2\widetilde{v}^{\frac{n+2\sigma}{n-2\sigma}}+\beta \widetilde{v}^{\frac{2\sigma}{n-2\sigma}}\widetilde{u}^{\frac{n}{n-2\sigma}} &\quad&\mbox{in }\ \ \R^{n},
\end{aligned}
\end{cases}
\]
and $(\widetilde{u}+\widetilde{v})(0)=1$. By Liouville Theorem (Theorem \ref{thm1}), we have
\be\label{eq:cl5}
(\widetilde{u}+\widetilde{v})(x)= \left(\frac{1}{1+|x|^2}\right)^{\frac{n-2\sigma}{2}}
\ee
modulo some multiple, scaling and translation.

On the other hand, we are going to show that  for any $\lda>0$, $x\in \R^n$,
\be\label{eq:aim1}
(\widetilde{u}+\widetilde{v})_{ x,\lda}(y)\leq (\widetilde{u}+\widetilde{v})(y)\quad\mbox{in } \R^n\backslash B_{\lda}(x).
\ee
By an elementary calculus lemma Proposition \ref{CHANG2} implies that
\[
\widetilde{u}+\widetilde{v}\equiv constant.
\]
This contradicts to \eqref{eq:cl5}.

In order to prove \eqref{eq:aim1}, it suffices to prove that for any $x\in \R^n$, $\lda>0$,
\be\label{eq:aim11}
(u_{j}+v_{j})_{x,\lda}(y)\leq (u_j+v_j)(y)\quad \mbox{in}\ \ \partial'(\Omega_j\backslash \B^+_{\lda}(X)).
\ee
Sending $j\to \infty$, \eqref{eq:aim1} follows.

Hence, let us arbitrarily fix $x_0\in\R^n$, $X_0=(x_0,0)$ and $\lda_0>0$. Then for all $j$ large, we have $|x_0|<\frac{R_j}{10}, 0<\lda_0<\frac{R_j}{10}$. If we have proved that for any $\lda\in(0,\lda_0)$,
\be\label{PW}
(U_{j}+V_{j})_{X_0,\lda}(\xi)\leq (U_j+V_j)(\xi)\quad \mbox{in}\ \  \Omega_j\backslash \B^+_{\lda}(X).
\ee
Together with the arbitrariness of  $x_0$ and $\lda_0$, \eqref{eq:aim11} has been verified.

For simplicity, we denote
\[
H_j(\xi):=(U_j+V_j)(\xi),\quad H_{j,X_0,\lambda}:=(U_{j}+V_{j})_{X_0,\lda}(\xi),
\]
and define
\[
\bar \lda(x):=\sup \big\{\mu\in (0,\lda_0)\big|H_{j,X_0,\lda}(\xi)\le H_j(\xi)\ \ \mbox{in }\ \om_j\backslash \B^+_{\lda}(X_0),~\forall\lda\in(0, \mu)\big\}.
\]
From what has been discussed above, it is clearly to know that if we  get $\bar \lda(x)=\lda_0$, then \eqref{PW} follows.
Therefore,  we need the following Lemma \ref{Lemma5.1}  to make sure that $\bar \lda(x)$ is well defined, and then we shall prove $\bar \lda(x)=\lda_0$. Before that, by a calculation it is easy to see that
\be
\begin{cases}
\begin{aligned}
&\mathrm{div}(t^{1-2\sigma} \nabla U_j)=0 &\quad&\mbox{in }\ \ \om_j,\\
&\mathrm{div}(t^{1-2\sigma} \nabla V_j)=0 &\quad&\mbox{in }\ \ \om_j,\\
&\frac{\partial U_j}{\partial \nu^\sigma} =\alpha_1u_j^{\frac{n+2\sigma}{n-2\sigma}}+\beta u_j^{\frac{2\sigma}{n-2\sigma}}v_j^{\frac{n}{n-2\sigma}} &\quad&\mbox{on }\ \ \partial'\om_j,\\
&\frac{\partial V_j}{\partial \nu^\sigma} =\alpha_2v_j^{\frac{n+2\sigma}{n-2\sigma}}+\beta v_j^{\frac{2\sigma}{n-2\sigma}}u_j^{\frac{n}{n-2\sigma}} &\quad&\mbox{on }\ \ \partial'\om_j.
\end{aligned}
\end{cases}
\ee

\begin{lem}\label{Lemma5.1} We would like to show that there exists $\lda_1\in (0,\lda_0)$ such that for any $\lda\in(0,\lda_1)$,
\be\label{eq:bigger}
H_{j,X_0,\lda}(\xi)\le H_j(\xi)\quad\mbox{\rm in }\ \ \om_j\backslash \B^+_{\lda}(X_0).
\ee
\end{lem}
Then we give that

\begin{lem}\label{Lemma5.2}
\be\label{Weq:gseq}
\bar \lda(x)=\lda_0.
\ee
\end{lem}

\begin{proof}[Proof of Lemma \ref{Lemma5.1}] There consists two steps. The step 1 we need to prove
\[
U_{j,X_0,\lda }(\xi)\leq U_j(\xi),\ \ V_{j,X_0,\lda }(\xi)\leq V_j(\xi) \quad \mbox{in }\ \ \overline{\B^+_\mu(X_0)}\backslash\B^+_\lda(X_0).
\]
This step follows exactly the same as that for the the proof of \eqref{Q} in  Lemma \ref{lemma1}. It follows that
\be
H_{j,X_0,\lda }(\xi)\leq H_j(\xi)\quad \mbox{in }\ \ \overline{\B^+_\mu(X)}\backslash\B^+_\lda(X).
\ee
Pay attention to the difference is that we choose
$$\lambda_0(x)=\mu\min\left\{\left(\frac{\inf\limits_{\partial'' \B^+_{\mu}(X)} U_j}{\sup\limits_{ \B^+_{\mu}(X)}U_j}\right)^{\frac{1}{n-2\sigma}},\left(\frac{\inf\limits_{\partial'' \B^+_{\mu}(X)} V_j}{\sup\limits_{ \B^+_{\mu}(X)}V_j}\right)^{\frac{1}{n-2\sigma}},\left(\frac{\inf\limits_{\partial'' \B^+_{\mu}(X)} H_j}{\sup\limits_{ \B^+_{\mu}(X)}H_j}\right)^{\frac{1}{n-2\sigma}} \right\} ,$$
instead of \eqref{IDA}.

The  step 2 we are going to prove
\[
H_{j,X_0,\lda }(\xi)\leq H_j(\xi)\quad \mbox{in }\ \ \R^{n+1}_+\backslash\overline{\B^+_\mu(X)}.
\]
Let $\phi(\xi):=\left(\frac{\mu}{|\xi-X_0|}\right)^{n-2\sigma}\inf\limits_{\partial'' \B^+_{\mu}(X_0)} H_j$,
which satisfies
\[
\begin{cases}
\begin{aligned}
&\mathrm{div}(t^{1-2\sigma}\nabla \phi)=0&\quad& \mbox{in }\ \R^{n+1}_+\setminus \B^+_{\mu}(X_0),\\
&\frac{\partial \phi}{\partial \nu^\sigma}=0&\quad& \mbox{on }\ \partial'(\R^{n+1}_+\setminus \B^+_{\mu}(X_0))
 \end{aligned}
\end{cases}
\]
and
\[
\phi(\xi)=\inf\limits_{\partial'' \B^{+}_{\mu}(X_0)} H_j\leq H_j(\xi)\quad\mbox{on }\ \partial'' \B^{+}_{\mu}(X_0).
\]
In addition, since $u+v\ge 1/C>0$ on $\pa B_1$, it follows from Proposition \ref{Harnack inequality} that
\be\label{eq:cl20}
H_{j}(\xi)\geq \frac{1}{Cw(\bar x_j)}>0 \quad \mbox{on }\pa'' \om_j.
\ee
 Since $\frac{|x_j|}{2}\le |\bar x_j|\leq \frac{3|x_j|}{2}<<1$, for any $\xi\in \pa'' \om_j$, i.e., $\left|\bar X_j+\frac{\xi}{w(\bar x_j)^{\frac{2}{n-2\sigma}}}\right|=1$, we have
\[
|\xi|\approx w(\bar x_j)^{\frac{2}{n-2\sigma}}.
\]
Thus
\be\label{eq:cl21}
H_{j}(\xi)\geq \frac{1}{Cw(\bar x_j)}>\left(\frac{\mu}{|\xi-X_0|}\right)^{n-2\sigma}\inf\limits_{\partial'' \B_{\mu}^+(X_0)} H_{j} \quad \mbox{on }\pa'' \om_j,
\ee
where we used the fact that $H_{j}\rightarrow H:=\widetilde{U}+\widetilde{V}$ locally uniformly

By the standard maximum principle argument,  we have
\be\label{9o}
H_j(\xi)\geq \left(\frac{\mu}{|\xi-X_0|}\right)^{n-2\sigma}\inf_{\partial'' \B^{+}_{\mu}(X_0)} H_j\ \  ~\mbox{in }\ \ \om_j\backslash \B^+_{\mu}(X_0).
\ee
Then for all $\xi\in \om_j\backslash \B^+_{\mu}(X_0)$, $0<\lda<\lda_{0}$, by \eqref{9o}, we have
\[
\begin{split}
H_{j,X_0,\lda }(\xi)&= \left(\frac{\lda}{|\xi-X_0|}\right)^{n-2\sigma}H_j(X_0+\frac{\lda^2(\xi-X_0)}{|\xi-X_0|^2})
\leq \left(\frac{\lda_0}{|\xi-X_0|}\right)^{n-2\sigma}\sup\limits_{\B^{+}_{\mu}(X_0)}H_j\\
&\leq \left(\frac{\mu}{|\xi-X_0|}\right)^{n-2\sigma}\inf\limits_{\partial ''\B^{+}_{\mu}(X_0)}H_j\leq H_j(\xi).
\end{split}
\]
\end{proof}
\begin{proof}[Proof of Lemma \ref{Lemma5.2}] We argue by contradiction. Suppose that $\bar\lda(x)<\lda_0$, we want to show that there exists a positive constant $\va$ such that for all $\lda\in(\bar\lda(x),\bar\lda(x)+\va)$,
\be\label{Zhong1}
H_{j,X_0,\lda}(\xi)\leq H_j(\xi)\quad \mbox{in }\ \ \om_j\setminus\B^+_{\lda}(X),
\ee
which contradicts with the definition of $\bar\lda(x)$.

Let us divide the region  $\om_j\setminus\B^+_{\lda}(X)$ into two parts. For $\delta, \delta_1>0$ small, which will be fixed later, denote
\[
\begin{array}{ll}
K_{1}:=\{\xi\in\om_j| 0<|\xi-Y_0|\leq\delta_1\},\vspace{0.2cm}\\
K_{2}:=\{\xi\in\om_j||\xi-Y_0|\geq\delta_1, |\xi-X_0|\geq\bar\lda(x)+\delta\},\vspace{0.2cm}\\
K_{3}:=\{\xi\in\om_j| \lda\leq|\xi-X_0|\leq\bar\lda(x)+\delta\},
\end{array}
\]
where $Y_0:=-\bar x_jw(\bar x_j)^{\frac{2}{n-2\sigma}}$.  In order to obtain \eqref{Zhong1} it  suffices to prove that it established on $K_{1}$, $K_2$, $K_3$. The following we will  prove \eqref{Zhong1} holds in $K_1$, $K_2$, $K_3$ respectively.

Similar to \eqref{eq:cl21}, we have
\[
H_{j}(\xi)\geq \frac{1}{Cw(\bar x_j)}>\left(\frac{\lda_0}{|\xi-X_0|}\right)^{n-2\sigma}\sup\limits_{ \B^+_{\lda_0}(X_0)} H_{j}\geq H_{j,X_0,\bar\lda(x)}(\xi) \quad \mbox{on }\pa'' \om_j,
\]
It follows from the strong maximum principle,
\[
H_{j,X_0,\bar\lda(x)}(\xi)< H_{j}(\xi)\quad\mbox{in}\ \ \overline \om_j\setminus \overline{\B^+_{\bar\lda(x)}(X_0)}
\]
and by Proposition \ref{prop:liminf}, there exist two positive constants $\delta_1$, $C_1$ such that
\[
H_{j}(\xi)-H_{j,X_0,\bar\lda(x)}(\xi)>C_1  \ \text{ in }\  K_{1}.
\]
Choose a  positive constant $\va_1$ small such that for all $\lda\in (\bar\lda(x),\bar\lda(x)+\va_1)$,
\[
H_{j,X_0,\bar\lda(x)}(\xi)-H_{j,X_0,\lda}(\xi)>-C_1/2\ \text{ in }\  K_{1}.
\]
Hence,
\be\label{E1}
H_{j}(\xi)-H_{j,X_0,\lda}(\xi)>C_1/2 \ \text{ in } \ K_{1}.
\ee
Together with
\[
H_{j,X_0,\bar\lda(x)}(\xi)< H_{j}(\xi)\quad\mbox{in}\ K_2,
\]
and the compactness of $K_2$, there exists a positive constant $C_2$ such that
\[
H_{j}(\xi)-H_{j,X_0,\bar\lda(x)}(\xi)>C_2  \ \text{ in }\  K_{2}.
\]
By the uniform continuity of $H_j$ on compact sets, there exists a  positive constant $\va_2$ small such that for all $\lda\in (\bar\lda(x),\bar\lda(x)+\va_2)$,
\[
H_{j,X_0,\bar\lda(x)}(\xi)-H_{j,X_0,\lda}(\xi)>-C_2/2\ \text{ in }\  K_{2}.
\]
Therefore,
\be\label{E1o}
H_{j}(\xi)-H_{j,X_0,\lda}(\xi)>C_2/2\ \text{ in } \ K_{2}.
\ee
Now let us focus on the region $K_3$. Using the narrow domain technique as that in Lemma \ref{lemma1}, we can choose $\delta$ small ( notice that we can choose $\va$ as small as we want ) such that
\[
 U_{j,X_0,\lda}(\xi)\leq U_{j}(\xi), \ V_{j,X_0,\lda}(\xi)\leq  V_{j}(\xi) \ \text{ in } \ K_3.
\]
Thus,
\be\label{E2}
 H_{j,X_0,\lda}(\xi)\leq H_{j}(\xi) \ \text{ in } \ K_2.
\ee
Combining \eqref{E1}, \eqref{E1o} with \eqref{E2}, we obtain that there exists a positive constant $\va_1$ such that for all $\lda\in(\bar\lda(x),\bar\lda(x)+\va)$,
\[
H_{j,X_0,\lda}(\xi)\leq H_{j}(\xi)\quad\mbox{in }\ \om_j\backslash \B_{\lda}^+(X_0),
\]
which contradicts with the definition of $\bar\lda(x)$. Lemma \ref{Lemma5.2} is proved.
\end{proof}
\end{proof}

One consequence of this upper bound is that every solution $(U,V)$ of \eqref{91} satisfies the following Harnack inequality, which will be used very frequently in this rest of the paper.
\begin{lem}\label{lem:spherical harnack}
Suppose that $(U,V)\in C^2(\B^+_1)\cap C^1(\overline{\B^+_1}\backslash\{0\})$ is a nonnegative solution of
\[
\begin{cases}
\begin{aligned}
&\mathrm{div}(t^{1-2\sigma} \nabla U)=0 &\quad&\mbox{\rm in }\ \ \B_1^{+},\\
&\mathrm{div}(t^{1-2\sigma} \nabla V)=0 &\quad&\mbox{\rm in }\ \ \B_1^{+},\\
&\frac{\partial U}{\partial \nu^\sigma} =\alpha_1u^{\frac{n+2\sigma}{n-2\sigma}}+\beta u^{\frac{2\sigma}{n-2\sigma}}v^{\frac{n}{n-2\sigma}} &\quad&\mbox{\rm on }\ \pa'\B^{+}_{1}\backslash\{0\},\\
&\frac{\partial V}{\partial \nu^\sigma} =\alpha_2v^{\frac{n+2\sigma}{n-2\sigma}}+\beta v^{\frac{2\sigma}{n-2\sigma}}u^{\frac{n}{n-2\sigma}} &\quad&\mbox{\rm on }\  \pa'\B^{+}_{1}\backslash\{0\}.
\end{aligned}
\end{cases}
\]
Then for all $0<r<1/4$, we have
\be\label{eq:spherical harnack}
\sup_{\B^+_{2r}\setminus\overline{\B^+_{r}}} (U+V)\le C \inf_{\B^+_{2r}\setminus\overline{\B^+_{r}}} (U+V),
\ee
where $C$ is a positive constant independent of $r$.
\end{lem}
\begin{proof}
Let
\[
U_1(X):=r^{\frac{n-2\sigma}{2}}U(rX),\ \ V_1(X):=r^{\frac{n-2\sigma}{2}}V(rX).
\]
It follows that
\[
\begin{cases}
\begin{aligned}
&\mathrm{div}(t^{1-2\sigma} \nabla U_1)=0 &\quad&\mbox{in }\ \ \B_4^{+}\backslash \B_{1/4}^{+},\\
&\mathrm{div}(t^{1-2\sigma} \nabla V_1)=0 &\quad&\mbox{in }\ \ \B_4^{+}\backslash \B_{1/4}^{+},\\
&\frac{\partial U_1}{\partial \nu^\sigma}=\alpha_1u_1^{\frac{n+2\sigma}{n-2\sigma}}+\beta u_1^{\frac{2\sigma}{n-2\sigma}}v_1^{\frac{n}{n-2\sigma}} &\quad&\mbox{on }\ \partial'(\B_4^{+}\backslash \B_{1/4}^{+}),\\
&\frac{\partial V_1}{\partial \nu^\sigma}=\alpha_2v_1^{\frac{n+2\sigma}{n-2\sigma}}+\beta v_1^{\frac{2\sigma}{n-2\sigma}}u_1^{\frac{n}{n-2\sigma}} &\quad&\mbox{on }\ \partial'(\B_4^{+}\backslash \B_{1/4}^{+}).
\end{aligned}
\end{cases}
\]
Next, we are going to prove
\[
|U_1(x,0)|,\ \ |V_1(x,0)|\le C\quad\mbox{on } \ \ \partial'(\B_4^{+}\backslash \B_{1/4}^{+}),
\]
where $C$ is a positive constant depending on $U$, $V$ but independent of $r$.

Indeed, by Theorem \ref{thm2}, there exist positive constants $\varepsilon$ and $C_1$ such that  for $|rx|<\varepsilon$,
\[
|rx|^{\frac{n-2\sigma}{2}}u(rx)\leq C_1,
\]
that is,
\[
|r|^{\frac{n-2\sigma}{2}}u(rx)\leq C_1|x|^{-\frac{n-2\sigma}{2}}\leq C.
\]
For $\varepsilon\leq|rx|\leq 1$, it follows from the continuity of $u$ that there exists a positive constant $C_2$ such that
\[
|r|^{\frac{n-2\sigma}{2}}u(rx)\leq C_2.
\]
It follows  that
\[
|U_1(x,0)|\leq C\quad\mbox{on } \ \ \partial'(\B_4^{+}\backslash \B_{1/4}^{+}),
\]
where $C$ is a positive constant depending on $U$ but independent of $r$.

With the help of Proposition \ref{Harnack inequality}, we have
\[
\sup_{1\le|X|\le 2} (U_1+ V_1)(X)\le C \inf_{1\le|X|\le 2} (U_1+ V_1)(X),
\]
that is,
\[
\sup_{1\le|X|\le 2} r^{\frac{n-2\sigma}{2}}(U+V)(rX)\le C \inf_{1\le|X|\le 2} r^{\frac{n-2\sigma}{2}}(U+V)(rX).
\]
Then we deduce that
\[
\sup_{r\le|X|\le 2r} (U+V)(X)\le C \inf_{r\le|X|\le 2r}(U+V)(X).
\]
where $C$ is another positive constant independent of $r$. Hence, \eqref{eq:spherical harnack} follows.
\end{proof}
\subsection{A Lower Bound and Removability}
We define the Pohozaev integral as
\[
\begin{split}
P(U,V,R)&=\frac{n-2\sigma}{2}\int_{\pa'' \B^+_{R}}t^{1-2\sigma}\left(\frac{\pa U}{\pa\nu} U+\frac{\pa V}{\pa\nu} V\right)
-\frac{R}{2}\int_{\pa'' \B^+_{R}}t^{1-2\sigma}\left(|\nabla U|^2+|\nabla V|^2\right)\\
&\quad+R\int_{\pa'' \B^+_{R}}t^{1-2\sigma}\left(\left|\frac{\pa U}{\pa\nu}\right|^2+\left|\frac{\pa V}{\pa\nu}\right|^2\right)
+\frac{R}{2^*}\int_{\pa B_{R}}\alpha_1u^{2^*}+\alpha_2v^{2^*}+2\beta u^{\frac{2^*}{2}}v^{\frac{2^*}{2}},
\end{split}
\]
where $2^*=\frac{2n}{n-2\sigma}$, $u(\cdot)=U(\cdot, 0)$ and $v(\cdot)=V(\cdot, 0)$ and $\nu$ is the unite outer normal of $\partial \B^+_{R}$.

Since
\begin{equation}\label{EP}
\begin{cases}
\begin{aligned}
&\mathrm{div}(t^{1-2\sigma} \nabla U)=0 &\quad&\mbox{\rm in }\ \ \B_1^{+},\\
&\mathrm{div}(t^{1-2\sigma} \nabla V)=0 &\quad&\mbox{\rm in }\ \ \B_1^{+},\\
&\frac{\partial U}{\partial \nu^\sigma} =\alpha_1u^{\frac{n+2\sigma}{n-2\sigma}}+\beta u^{\frac{2\sigma}{n-2\sigma}}v^{\frac{n}{n-2\sigma}} &\quad&\mbox{\rm on }\ \pa'\B^{+}_{1}\backslash\{0\},\\
&\frac{\partial V}{\partial \nu^\sigma} =\alpha_2v^{\frac{n+2\sigma}{n-2\sigma}}+\beta v^{\frac{2\sigma}{n-2\sigma}}u^{\frac{n}{n-2\sigma}} &\quad&\mbox{\rm on }\  \pa'\B^{+}_{1}\backslash\{0\}.
\end{aligned}
\end{cases}
\end{equation}
Multiplying the first equation in \eqref{EP} by $(U_{X,\lda }-U)^+$, the second equation by $(V_{X,\lda }-V)^+$ and integrating by parts in $\B_{R}^+(X)\backslash \overline{\B_{S}^+(X)}$. By a direct calculations, we  obtain that
\[
P(U,V,R)=P(U,V,S),
\]
which implies that $P(U,V,R)$ is a constant independent of $R$.
\begin{thm}\label{14}
If
\[
\liminf_{|x|\to 0}|x|^{\frac{n-2\sigma}{2}}(u+v)(x)=0,
\]
then
\[
\lim_{|x|\to 0}|x|^{\frac{n-2\sigma}{2}}u(x)=\lim_{|x|\to 0}|x|^{\frac{n-2\sigma}{2}}v(x)=0.
\]
\end{thm}

\begin{proof}
We just need to prove
\[
\limsup_{|x|\to 0}|x|^{\frac{n-2\sigma}{2}}(u+v)(x)=0,
\]
then
\[
\limsup_{|x|\to 0}|x|^{\frac{n-2\sigma}{2}}u(x)=0,\ \
\limsup_{|x|\to 0}|x|^{\frac{n-2\sigma}{2}}v(x)=0.
\]
The result is to obtain that
\[
\lim_{|x|\to 0}|x|^{\frac{n-2\sigma}{2}}u(x)=\lim_{|x|\to 0}|x|^{\frac{n-2\sigma}{2}}v(x)=0.
\]
We suppose by contradiction that
\be\label{JIA}
\limsup_{|x|\to 0}|x|^{\frac{n-2\sigma}{2}}(u+v)(x)=C>0.
\ee
Hence, there exist two sequences of points $\{x_i\}, \{y_i\}$ satisfying
\[
x_i\to 0,\quad y_i\to 0\quad \mbox{as }i\to\infty,
\]
such that
\[
|x_i|^{\frac{n-2\sigma}{2}}(u+v)(x_i)\to 0,\quad |y_i|^{\frac{n-2\sigma}{2}}(u+v)(y_i)\to C>0\quad \mbox{as }i\to\infty.
\]
Define
\[
w(r):=r^{\frac{n-2\sigma}{2}}(\overline{u+v})(r),
\]
where $(\overline{u+v})(r):=\frac{1}{|\partial B_r|}\int _{\partial B_r}u+vds$ is the spherical average of $u+v$ on $\pa B_r$.

First,  we are going to prove that
\be\label{CL1}
w(r)\leq C,\quad r\in(0,1/2).
\ee
\be\label{CL2}
\liminf_{r\to 0}w(r)=0.
\ee
\be\label{CL3}
\limsup_{r\to 0}w(r)>0.
\ee
We follow the usual convention of denoting by $C$ a general positive constant that may vary from line to line.

\emph{Proof of \eqref{CL1}}: From Theorem \ref{thm2} and the continuous of $u$ and $v$, there exists a positive constant $C$ such that
$$|x|^{\frac{n-2\sigma}{2}}u(x),\ \ |x|^{\frac{n-2\sigma}{2}}v(x)\leq C\quad\mbox{in}\ B_{1/2}\backslash\{0\},$$ then by Lemma \ref{lem:spherical harnack}, for any $x\in B_{1/2}\backslash\{0\}$,
\[
\begin{aligned}
w(r):=&w(|x|)=r^{\frac{n-2\sigma}{2}}\frac{1}{|\partial B_r|}\int _{\partial B_r}u+vds\leq r^{\frac{n-2\sigma}{2}}\sup_{\partial B_r}(u+v)\\
\leq &Cr^{\frac{n-2\sigma}{2}}\inf_{\partial B_r}(u+v)\leq C|x|^{\frac{n-2\sigma}{2}}(u+v)(x)\leq C.
\end{aligned}
\]

\emph{Proof of \eqref{CL2}}: Choose $x_i\rightarrow 0$ which satisfies $|x_i|^{\frac{n-2\sigma}{2}}(u+v)(x_i)\to 0$. It follows that
\[
\begin{aligned}
w(r_i)=&w(|x_i|)=r_i^{\frac{n-2\sigma}{2}}\frac{1}{|\partial B_{r_i}|}\int _{\partial B_{r_i}}u+vds\leq r_i^{\frac{n-2\sigma}{2}}\sup_{\partial B_{r_i}}(u+v)\\
\leq& Cr_i^{\frac{n-2\sigma}{2}}\inf_{\partial B_{r_i}}(u+v)\leq C|x_i|^{\frac{n-2\sigma}{2}}(u+v)(x_i)\rightarrow 0.
\end{aligned}
\]

\emph{Proof of \eqref{CL3}:} We suppose the contrary that $\limsup_{r\to 0}w(r)=0$, implies that $\lim_{r\to 0}w(r)=0$. It follows that for any positive constant $\varepsilon$, there exists  $\delta(\varepsilon)>0$, so that for any $r_i:=|x_i|< \delta$,
\[
|x_i|^{\frac{n-2\sigma}{2}}(\overline{u+v})(x_i)<\varepsilon.
\]
We deduce that
\[
\begin{aligned}
r_i^{\frac{n-2\sigma}{2}}(u+v)(x_i)\leq &r_i^{\frac{n-2\sigma}{2}}\sup_{\partial B_{r_i}}(u+v)\leq C r_i^{\frac{n-2\sigma}{2}}\inf_{\partial B_{r_i}}(u+v)\\
\leq& Cr_i^{\frac{n-2\sigma}{2}}\frac{1}{|\partial B_{r_i}|}\int _{\partial B_{r_i}}u+vds=Cr_i^{\frac{n-2\sigma}{2}}(\overline{u+v})(r_i)\\
\leq& C\varepsilon.
\end{aligned}
\]
from Lemma \ref{lem:spherical harnack}. By the arbitrariness of  $\varepsilon$, we conclude that
\[
\lim_{x\to 0}|x|^{\frac{n-2\sigma}{2}}(u+v)(x)=0.
\]
It is a contradiction with \eqref{JIA}.

A combination of  \eqref{CL1}, \eqref{CL2} and  \eqref{CL3} yields there exists a sequence of positive numbers $\{r_i\}$ converging to $0$ such that
\[
r_i^{\frac{n-2\sigma}{2}}(\overline{u+ v})(r_i)\to 0\quad\mbox{as }i\to\infty,
\]
and $r_i$  are local minimum of $r^{\frac{n-2\sigma}{2}}(\overline{u+ v})(r)$ for every $i$.
Let
\[
W_i(X):=\frac{U(r_iX)}{(U+V)(r_ie_1)},\ \ Z_i(X):=\frac{V(r_iX)}{(U+V)(r_ie_1)},
\]
where $e_1=(1,0,\cdots, 0)\in\R^{n+1}$. A direct calculation gives that
\[
\begin{cases}
\begin{aligned}
&\mathrm{div}(t^{1-2\sigma} \nabla W_i)=0 & \quad &\mbox{in }\R^{n+1}_+,\\
&\mathrm{div}(t^{1-2\sigma} \nabla Z_i)=0 & \quad &\mbox{in }\R^{n+1}_+,\\
&\frac{\pa W_i}{\pa \nu^\sigma} =( r_i^{\frac{n-2\sigma}{2}}(U+V)(r_ie_1))^\frac{4\sigma}{n-2\sigma}\left(\alpha_1w_i^{\frac{n+2\sigma}{n-2\sigma}}+\beta w_i^{\frac{2\sigma}{n-2\sigma}}z_i^{\frac{n}{n-2\sigma}}\right)&\quad &\mbox{on }\partial'\R^{n+1}_+\setminus\{0\},\\
&\frac{\pa Z_i}{\pa \nu^\sigma} =( r_i^{\frac{n-2\sigma}{2}}(U+V)(r_ie_1))^\frac{4\sigma}{n-2\sigma}\left(\alpha_2z_i^{\frac{n+2\sigma}{n-2\sigma}}+\beta z_i^{\frac{2\sigma}{n-2\sigma}}w_i^{\frac{n}{n-2\sigma}}\right)&\quad &\mbox{on }\partial'\R^{n+1}_+\setminus\{0\},\\
\end{aligned}
\end{cases}
\]
where $w_i(x):=W_i(x,0)$, $z_i(x):=Z_i(x,0)$. With the help of Harnack inequality (Lemma \ref{lem:spherical harnack}), it follows that
\[
\begin{aligned}
\sup _{\B^+_{R}\setminus\overline{\B^+_{1/R}}}(W_i+Z_i)(X)&=\sup _{\B^+_{R}\setminus\overline{\B^+_{1/R}}}\frac{(U+V)(r_iX)}{(U+V)(r_ie_1)}
\leq C\inf _{\B^+_{R}\setminus\overline{\B^+_{1/R}}}\frac{(U+V)(r_iX)}{(U+V)(r_ie_1)}\\
&\leq C\inf _{|X|=1}\frac{(U+V)(r_iX)}{(U+V)(r_ie_1)}
\leq C\frac{(U+V)(r_ie_1)}{(U+V)(r_ie_1)}\leq C.
\end{aligned}
\]
Combining with $W_i(X)>0$, $Z_i(X)>0$ , we conclude that $W_i(X)$, $Z_i(X)$ is locally uniformly bounded away from the origin.
Using the Harnack inequality (Lemma \ref{lem:spherical harnack}) again, we have
\[
\begin{aligned}
r_i^{\frac{n-2\sigma}{2}}(U+V)(r_ie_1)=&r_i^{\frac{n-2\sigma}{2}}\frac{1}{|\partial B_{r_i}|}\int_{\partial B_{r_i}}(U+V)(r_ie_1)dx\\
\leq & r_i^{\frac{n-2\sigma}{2}}\frac{1}{|\partial B_{r_i}|}\int_{\partial B_{r_i}}\sup_{\partial B_{r_i}}(u+v)dx\\
\leq  &Cr_i^{\frac{n-2\sigma}{2}}\frac{1}{|\partial B_{r_i}|}\int_{\partial B_{r_i}}\inf_{\partial B_{r_i}}(u+v)dx\\
\leq & Cr_i^{\frac{n-2\sigma}{2}}\frac{1}{|\partial B_{r_i}|}\int_{\partial B_{r_i}}u+vdx\\
=&Cr_i^{\frac{n-2\sigma}{2}}( \overline{u+v})(r_i).
\end{aligned}
\]
Then
\be\label{infty}
r_i^{\frac{n-2\sigma}{2}}((U+V)(r_ie_1))\to 0\quad \mbox{as } i\to\infty.
\ee
By \cite[Proposition 2.3, 2.6, 2.8]{JLX}, there exists some $\al>0$ such that for any $0<r<1<R$,
\[
\|W_j+Z_j\|_{W^{1,2}(t^{1-2\sigma}, \B^+_R\setminus\overline \B^+_r)}+\|W_j+Z_j\|_{C^\al( \B^+_R\setminus\overline \B^+_r)}+\|w_j+z_j\|_{C^{2,\al}( B_R\setminus\overline B_r)}\le C(R,r).
\]
Since $W_j$, $Z_j\geq0$, it follows that
\[
\|W_j\|_{W^{1,2}(t^{1-2\sigma}, \B^+_R\setminus\overline \B^+_r)}+\|W_j\|_{C^\al( \B^+_R\setminus\overline \B^+_r)}+\|w_j\|_{C^{2,\al}( B_R\setminus\overline B_r)}\le C(R,r),
\]
and
\[
\|Z_j\|_{W^{1,2}(t^{1-2\sigma}, \B^+_R\setminus\overline \B^+_r)}+\|Z_j\|_{C^\al( \B^+_R\setminus\overline \B^+_r)}+\|z_j\|_{C^{2,\al}( B_R\setminus\overline B_r)}\le C(R,r),
\]
where $C(R,r)$ is independent of $i$. Then up to a subsequence, $\{W_i\}$, $\{Z_i\}$  converges to a nonnegative function $W$, $Z\in W_{\rm loc}^{1,2}(t^{1-2\sigma}, \overline{\R^{n+1}_+}\setminus\{0\})\cap C^{\al}_{\rm loc}(\overline{\mathbb{R}^{n+1}_+}\setminus\{0\})$ respectively, and
\[
\begin{cases}
\begin{aligned}
&\mathrm{div}(t^{1-2\sigma} \nabla W)=0 & \quad &\mbox{in }\R^{n+1}_+,\\
&\mathrm{div}(t^{1-2\sigma} \nabla Z)=0 & \quad &\mbox{in }\R^{n+1}_+,\\
&\frac{\pa W}{\pa \nu^\sigma} = 0&\quad &\mbox{on }\partial'\R^{n+1}_+\setminus\{0\},\\
&\frac{\pa Z}{\pa \nu^\sigma} = 0&\quad &\mbox{on }\partial'\R^{n+1}_+\setminus\{0\}.
\end{aligned}
\end{cases}
\]
Through a B\^ocher type Theorem ( Proposition \ref{Bocher} ), we deduce that
\[
W(X)=\frac{a_1}{|X|^{n-2\sigma}}+b_1,\ \ Z(X)=\frac{a_2}{|X|^{n-2\sigma}}+b_2,
\]
where $a_1$, $b_1$, $a_2$, $b_2$ are nonnegative constants. Let $w(x):=W(x,0)$ and $z(x):=Z(x,0)$. We know that $w_i(x)\to w(x)$ in $C^2_{\rm loc}(\R^n\setminus\{0\})$ and $z_i(x)\to z(x)$ in $C^2_{ \rm loc}(\R^n\setminus\{0\})$. Besides, since $r_i$  are local minimum of $r^{\frac{n-2\sigma}{2}}(\overline{u+ v})(r)$ for every $i$, we conclude that $r^{\frac{n-2\sigma}{2}}\overline{w+z}(r)$  has a critical point at $r=1$, which implies that $a_1+a_2=b_1+b_2$. Furthermore, we deduce $a_1+a_2=\frac{1}{2}$ by $W(e_1)+Z(e_1)=1$. Now let us compute $P(U, V)$.

Before that, we will prove that
\begin{equation}\label{LJ1}
|\nabla_x U(X)|, \ |\nabla_x V(X)|\le o(1)r_i^{-\frac{n-2\sigma}{2}-1}\quad\mbox{for all }|X|=r_i,
\end{equation}
and
\begin{equation}\label{LJ2}
|t^{1-2\sigma} U_t(X)|, \ |t^{1-2\sigma}V_t(X)|\le o(1)r_i^{-\frac{n-2\sigma}{2}-2\sigma}\quad\mbox{for all }|X|=r_i.
\end{equation}
Like before, we just give a proof of $U$.

Indeed, it follows from Proposition \ref{Tidu} that $|\nabla_x W_i|$, $|\nabla_x Z_i|$ are locally uniformly bounded in $L_{\rm loc}^{\infty}(\overline{\R^{n+1}_+}\setminus\{0\})$ and
$|t^{1-2\sigma}\pa_t W_i|$,   $|t^{1-2\sigma}\pa_t Z_i|$ are locally uniformly bounded in $C_{\rm loc}^{\gamma}(\overline{\R^{n+1}_+}\setminus\{0\})$ for some $\gamma>0$. By the fact that
\be
|\nabla_x W_i|=\frac{|\nabla_x U(r_iX)|}{(U+V)(r_ie_1)}\leq C,
\ee
let $Y:=r_i X$, we have
\[
|\nabla_y U(Y)|\leq Cr_i^{-1}(U+V)(r_ie_1)=Cr_i^{-1-\frac{n-2\sigma}{2}}r_i^{\frac{n-2\sigma}{2}}(U+V)(r_ie_1).
\]
Together with \eqref{infty}, we deduce that
\[
|\nabla_y U(Y)|\le o(1)r_i^{-\frac{n-2\sigma}{2}-1}\quad\mbox{for all }|Y|=r_i.
\]
Then \eqref{LJ1} follows.

On the other hand,
\[
|t^{1-2\sigma}\pa_t W_i|=\frac{|t^{1-2\sigma}\pa_t U(r_iX)|}{(U+V)(r_ie_1)}\leq C,
\]
let $Y:=r_i X$, it follows that
\[
|Y_{n+1}^{1-2\sigma}\pa_{n+1} U(Y)|\le C r_i^{-2\sigma} (U+V)(r_ie_1)=Cr_i^{-2\sigma-\frac{n-2\sigma}{2}}r_i^{\frac{n-2\sigma}{2}}(U+V)(r_ie_1).
\]
Together with \eqref{infty}, we have
\[
|t^{1-2\sigma}U_t(X)|\le o(1)r_i^{-\frac{n-2\sigma}{2}-2\sigma}\quad\mbox{for all }|X|=r_i.
\]
Then \eqref{LJ2} follows.

Thus
\[
P(U,V)=\lim_{i\to\infty} P(U,V,r_i)=0.
\]
Since $P(U,V)$ is a constant, it follows that
\[
P(U,V,r_i)=0\quad\mbox{for all }i.
\]
In addition, for all $i$,
\begin{equation*}
\begin{split}
0&=P(U,V,r_i)=P(r_i^{\frac{n-2\sigma}{2}}U(r_iX), r_i^{\frac{n-2\sigma}{2}}V(r_iX),1)\\
&=P(r_i^{\frac{n-2\sigma}{2}}(U+V)(r_ie_1)W_i,r_i^{\frac{n-2\sigma}{2}}(U+V)(r_ie_1)Z_i,1).
\end{split}
\end{equation*}
As a consequence, we obtain that
\[
\begin{split}
0=&\frac{n-2\sigma}{2}\int_{\pa'' \B^+_{1}}t^{1-2\sigma}\left(\frac{\pa W_i}{\pa\nu} W_i+\frac{\pa Z_i}{\pa\nu} Z_i\right)
-\frac{1}{2}\int_{\pa'' \B^+_{1}}t^{1-2\sigma}\left(|\nabla W_i|^2+|\nabla Z_i|^2\right)\\
&+\int_{\pa'' \B^+_{1}}t^{1-2\sigma}\left(\left|\frac{\pa W_i}{\pa\nu}\right|^2+\left|\frac{\pa Z_i}{\pa\nu}\right|^2\right)\\
&+\frac{1}{2^*}\int_{\pa B_{1}}\left(r_i^{\frac{n-2\sigma}{2}}(U+V)(r_ie_1)\right)^{\frac{4\sigma}{n-2\sigma}}(\alpha_1W_i^{2^*}+\alpha_2Z_i^{2^*}+2\beta W_i^{\frac{2^*}{2}}Z_i^{\frac{2*}{2}}).
\end{split}
\]Sending $i\to\infty$, we have
\[
\begin{split}
0=&\frac{n-2\sigma}{2}\int_{\pa'' \B^+_{1}}t^{1-2\sigma}\left(\frac{\pa W}{\pa\nu} W+\frac{\pa Z}{\pa\nu} Z\right)
-\frac{1}{2}\int_{\pa'' \B^+_{1}}t^{1-2\sigma}\left(|\nabla W|^2+|\nabla Z|^2\right)\\
&+\int_{\pa'' \B^+_{1}}t^{1-2\sigma}\left(\left|\frac{\pa W}{\pa\nu}\right|^2+\left|\frac{\pa Z}{\pa\nu}\right|^2\right)
=C,
\end{split}
\]
if $a_1\cdot a_2=0$, together with $a_1+a_2=1/2$, without loss of generality,we assume that $a_1=0$ and $a_2=\frac{1}{2}$,  then $C=-\frac{(n-2\sigma)^2}{8}\int_{\pa'' \B^+_{1}}t^{1-2\sigma}$; otherwise,
$C=-\frac{(n-2\sigma)^2}{2}\int_{\pa'' \B^+_{1}}t^{1-2\sigma}a_1^2$, which is a contradiction. Then, we obtain
\[
\limsup_{|x|\to 0}|x|^{\frac{n-2\sigma}{2}}(u+v)(x)=0.
\]
\end{proof}
\begin{thm}\label{15}
If
\[
\lim_{|x|\to 0}|x|^{\frac{n-2\sigma}{2}}u(x)=\lim_{|x|\to 0}|x|^{\frac{n-2\sigma}{2}}v(x)=0
\]
then both $u$, $v$ can be extended as a  continuous function at the origin $0$.
\end{thm}
\begin{proof}
Via the Harnack inequality ( Lemma \ref{lem:spherical harnack} ), it is easy to see that
\[
\lim_{|X|\to 0}|X|^{\frac{n-2\sigma}{2}}U(X)=0,
\lim_{|X|\to 0}|X|^{\frac{n-2\sigma}{2}}V(X)=0.
\]
For $0<\mu\le n-2\sigma$ and $\delta>0$, let
\[
\Phi_\mu(X):=|X|^{-\mu}-\delta t^{2\sigma}|X|^{-(\mu+2\sigma)}.
\]
By a direct calculation, it follows that
\[
\mathrm{div}(t^{1-2\sigma}\nabla \Phi_\mu(X))=t^{1-2\sigma}|X|^{-(\mu+2)}\left(-\mu(n-2\sigma-\mu)+\frac{\delta(\mu+2\sigma)(n-\mu)t^{2\sigma}}{|X|^{2\sigma}}\right),
\]
and
\[
-\dlim_{t\rightarrow 0^+}t^{1-2\sigma}\pa_t \Phi_\mu(x,t)=2\delta\sigma|x|^{-(\mu+2\sigma)}= 2\delta\sigma|x|^{-2\sigma}\Phi_{\mu}(x,0).
\]
Let $\alpha\in (0,\frac{n-2\sigma}{2})$ be fixed, $\beta=\frac{n-2\sigma}{2}+1$ and $\Phi=C\Phi_\alpha+\va\Phi_\beta$, where $C, \va$ are positive constants. We can choose $\delta$ small ( depending on $\al$ ) and $\varepsilon$ small if needed such that
\[
\begin{cases}
\begin{aligned}
&\mathrm{div}(t^{1-2\sigma} \nabla \Phi)\le 0 & \quad &\mbox{in }\B^+_2,\\
&\frac{\pa \Phi}{\pa \nu^\sigma} = 2\delta\sigma|x|^{-2\sigma}\Phi(x,0)&\quad &\mbox{on }\partial'\B^+_2\setminus\{0\}.
\end{aligned}
\end{cases}
\]
Due to
\[
\lim_{|x|\to 0}|x|^{\frac{n-2\sigma}{2}}u(x)=\lim_{|x|\to 0}|x|^{\frac{n-2\sigma}{2}}v(x)=0,
\]
we can choose $\tau$ small to ensure that for all $0<|x|<\tau$,
\be\label{SS}
a(x)=(\alpha_1+\alpha_2+2\beta)(u+v)^{\frac{4\sigma}{n-2\sigma}}\le 2\delta\sigma|x|^{-2\sigma}.
\ee
Combining this, we obtain that
\begin{equation*}
\begin{split}
&2\delta\sigma|x|^{-2\sigma}\Phi(x,0)-\alpha_1u^{\frac{n+2\sigma}{n-2\sigma}}-\beta u^{\frac{2\sigma}{n-2\sigma}}v^{\frac{n}{n-2\sigma}}-\alpha_2v^{\frac{n+2\sigma}{n-2\sigma}}-\beta v^{\frac{2\sigma}{n-2\sigma}}u^{\frac{n}{n-2\sigma}}\\
\geq &2\delta\sigma|x|^{-2\sigma}\Phi(x,0)-(\alpha_1+\alpha_2+2\beta)(u+v)^{\frac{n+2\sigma}{n-2\sigma}}\\
\geq &2\delta\sigma|x|^{-2\sigma}\Phi(x,0)-2\delta\sigma|x|^{-2\sigma}(u+v)\\
= &2\delta\sigma|x|^{-2\sigma}(\Phi-U-V)(x,0)
\end{split}
\end{equation*}
Then we have
\[
\begin{cases}
\begin{aligned}
&\mathrm{div}(t^{1-2\sigma} \nabla (\Phi-U-V))\le 0 & \quad &\mbox{in }\B^+_\tau,\\
&\frac{\pa (\Phi-U-V)}{\pa \nu^\sigma}\ge 2\delta\sigma|x|^{-2\sigma}(\Phi-U-V)(x,0)&\quad &\mbox{on }\partial'\B^+_\tau\setminus\{0\}.
\end{aligned}
\end{cases}
\]

For every $\va>0$, we have that $\Phi\ge U+V$ near $0$. We can choose $C$ ( depending on $\al$ ) sufficiently large so that $\Phi\ge U+V$ on $\pa'' \B_\tau$. Hence, by Proposition \ref{A.2} (we can choose $\delta$ even smaller if needed), it follows that
\[
U+V\le \Phi\quad\quad\mbox{in }\ \ \B_\tau^+.
\]
After sending $\va\to 0$, we have
\[
 U+V\le  C(\al)\Phi_\alpha\le C(\al)|X|^{-\al}\quad\quad\mbox{in }\ \ \B_\tau^+.
\]
Because $U$, $V>0$, which implies  that
\be\label{O}
 U,V \le  C(\al)|X|^{-\al}\quad\quad\mbox{in }\ \ \B_\tau^+,
\ee
Furthermore, we want to prove that
\be\label{T}
 |\nabla_x U(X)|, |\nabla_x V(X)|\le C(\al)|X|^{-\al-1}\quad\quad\mbox{in }\ \ \B_\tau^+
\ee
and
\be\label{TT}
 |t^{1-2\sigma}\pa_t U(X)|,  |t^{1-2\sigma}\pa_t V(X)|\le C(\al)|X|^{-\al-2\sigma}\quad\quad\mbox{in }\ \ \B_\tau^+.
\ee

Indeed, for all $r\in (0,\tau)$, define
\[
\widetilde{U}(X):=r^{\al}U(rX),\ \ \widetilde{V}(X):=r^{\al}V(rX)\quad\quad\mbox{in }\ \ \B_1^{+}\backslash \B_{1/4}^{+},
\]
Via a direct calculation, we have
\be
\begin{cases}
\begin{aligned}
&\mathrm{div}(t^{1-2\sigma} \nabla \widetilde{U})=0 &\quad&\mbox{in }\ \ \B_1^{+}\backslash \B_{1/4}^{+},\\
&\frac{\partial \widetilde{U}}{\partial \nu^\sigma}=r^{2\sigma+\alpha-\frac{n+2\sigma}{n-2\sigma}\al}\left(\alpha_1\widetilde{u}^{\frac{n+2\sigma}{n-2\sigma}}+\beta \widetilde{u}^{\frac{2\sigma}{n-2\sigma}}\widetilde{v}^{\frac{n}{n-2\sigma}}\right) &\quad&\mbox{on }\ \ \partial'(\B_1^{+}\backslash \B_{1/4}^{+}).
\end{aligned}
\end{cases}
\ee
Combining with \eqref{O}, it follows that
\[
|\widetilde{U}(X)|=|r^{\al}U(rX)|\leq C(\al)r^{\al}|rX|^{-\al}\leq C(\al).
\]
By Proposition \ref{prop:guji1}, Proposition \ref{prop:guji} and Proposition \ref{Tidu}, we deduce that
\[
|\nabla_x \widetilde{U}|\leq C(\al),
\]
and
\[
|t^{1-2\sigma}\partial_t \widetilde{U}|\leq C(\al),
\]

Now, let $Y:=rX\in \B_\tau^+\setminus\{0\}$, it follows that
\[
|\nabla_x \widetilde{U}|=|r^{\al}\nabla_x U(rX)|=|r^{\al+1}\nabla_y U(Y)|\leq C(\al),
\]
which implies that
\begin{equation*}
|\nabla_y U(Y)|\leq C(\al)r^{-\al-1}=C(\al)|rX|^{-\al-1}|X|^{\al+1}\leq C(\al)|Y|^{-\al-1},
\end{equation*}
Then estimate \eqref{T} follows. Estimate \eqref{TT} follows by the similar argument.

A consequence of \eqref{T} and \eqref{TT} is
\[
U, \ \ V \in W^{1,2}(t^{1-2\sigma},\B^+_\tau).
\]Indeed, using $\al<\frac{n-2\sigma}{2}$, it is elementary to verify that
\[
\begin{split}
\int_{\B^+_\tau}t^{1-2\sigma}U^2dX\leq&C(\al)\int_0^{\tau}\int_{B_\tau}t^{1-2\sigma}|X|^{-2\al}dX\\
\leq&C(\al)\int_0^{\tau}t^{1-2\sigma}dt\int_{B_\tau}|x|^{-2\al}dx
< +\infty,
\end{split}
\]
and
\[
\begin{split}
\int_{\B^+_\tau}t^{1-2\sigma}|\nabla U|^2dX=&\int_{\B^+_{\tau}}t^{1-2\sigma}(|\nabla_x U(X)|^2+|\nabla_t U(X)|^2)dX\\
\leq&C(\al)\int_{\B^+_\tau}t^{1-2\sigma}|X|^{-2\al-2}dX+C(\al)\int_{\B^+_\tau}t^{2\sigma-1}|X|^{-2\al-4\sigma}dX
\end{split}
\]
Then,
\[
\begin{split}
\int_{\B^+_\tau}t^{1-2\sigma}|X|^{-2\al-2}dX\leq&\int_0^{\tau}t^{1-2\sigma}\int_{x^2+t^2\leq \tau^2}(|x|^2+t^2)^{-\al-1}dxdt\\
=&\int_{0}^{\tau}t^{1-2\sigma}\int_0^{\sqrt{\tau^2-t^2}}\rho^{n-1}(\rho^2+t^2)^{-\al-1}d\rho dt\\
=&\frac{1}{2}\int_{0}^{\tau}t^{1-2\sigma}\int_{t^2}^{\tau^2}\rho^{n-2}(\rho^2+t^2)^{-\al-1}d(\rho^2+t^2) dt\\
\leq&\int_{0}^{\tau}t^{1-2\sigma}\int_{t^2}^{\tau^2}(\rho^2+t^2)^{\frac{n-2}{2}-\al-1}d(\rho^2+t^2) dt\\
=&\left(\frac{n-2}{2}-\al\right)\left[\int_{0}^{\tau}t^{1-2\sigma}\tau^{n-2-2\al}dt-\int_{0}^{\tau}t^{n-1-2\al-2\sigma}dt\right]\\
\leq &C \tau^{n-2\sigma-2\al}<+\infty.
\end{split}
\]
\[
\begin{split}
\int_{\B^+_\tau}t^{2\sigma-1}|X|^{-2\al-4\sigma}dX\leq&\int_0^{\tau}t^{2\sigma-1}\int_{x^2+t^2\leq \tau^2}|X|^{-2\al-4\sigma}dX\\
=&\int_{0}^{\tau}t^{2\sigma-1}\int_0^{\sqrt{\tau^2-t^2}}\rho^{n-1}(\rho^2+t^2)^{-\al-2\sigma}d\rho dt\\
\leq&\int_0^{\tau}t^{2\sigma-1}\int_0^{\sqrt{\tau^2-t^2}}(\rho^2+t^2)^{\frac{n-2}{2}-\al-2\sigma}d(\rho^2+t^2) dt\\
=&\left(\frac{n}{2}-\al-2\sigma\right)\left(\tau^{n-2\al-4\sigma}\int_0^{\tau}t^{2\sigma-1}dt-\int_0^{\tau}t^{n-2\sigma-2\al-1}dt\right)\\
\leq &C \tau^{n-2\sigma-2\al}<+\infty.
\end{split}
\]
Hence, $$\int_{\B^+_\tau}t^{1-2\sigma}|\nabla U|^2dX\leq\infty.$$

On the other hand, for $\va>0$ small, let $\eta_\va$ be a smooth cut-off function satisfying $\eta\equiv 0$ in $\B_\va$, $\eta\equiv 1$ outside of $\B_{2\va}$ and $\nabla \eta_\va\le C\va^{-1}$. Let $\varphi\in C_c^{\infty}(\B_\tau^+\cup\pa'\B_\tau^+)$. Multiplying the equation in \eqref{BI} by $\varphi\eta_\va$ and integrating by parts lead to
\[
\int_{\B_\tau^+} t^{1-2\sigma}\nabla U\nabla (\varphi\eta_\va)=\int_{\pa'\B_\tau^+} \left(\alpha_1U^{\frac{n+2\sigma}{n-2\sigma}}+\beta U^{\frac{2\sigma}{n-2\sigma}}V^{\frac{n}{n-2\sigma}}\right)\varphi\eta_\va.
\]
From
\[
\left(\alpha_1U^{\frac{n+2\sigma}{n-2\sigma}}+\beta U^{\frac{2\sigma}{n-2\sigma}}V^{\frac{n}{n-2\sigma}}\right)\varphi\eta_\va\leq 2(\al_1+\beta)(U+V)^{\frac{n+2\sigma}{n-2\sigma}}\varphi,
\]
and
\[
\int_{\pa'\B_\tau^+}(U+V)^{\frac{n+2\sigma}{n-2\sigma}}\varphi\leq \left(\int_{\pa'\B_\tau^+}(U+V)^{\frac{2n}{n-2\sigma}}\right)^{\frac{n+2\sigma}{2n}}\left(\int_{\pa'\B_\tau^+}\varphi^{\frac{2n}{n-2\sigma}}\right)^{\frac{n-2\sigma}{2n}},
\]
which implies that
$$\int_{\pa'\B_\tau^+}(U+V)^{\frac{n+2\sigma}{n-2\sigma}}\varphi<\infty.$$
By the dominated convergence theorem and sending $\va\to 0$, we have
\[
\int_{\B_\tau^+} t^{1-2\sigma}\nabla U\nabla \varphi=\int_{\pa'\B_\tau^+} \left(\alpha_1U^{\frac{n+2\sigma}{n-2\sigma}}+\beta U^{\frac{2\sigma}{n-2\sigma}}V^{\frac{n}{n-2\sigma}}\right)\varphi.
\]
It follows that
\[
\begin{cases}
\begin{aligned}
&\mathrm{div}(t^{1-2\sigma} \nabla U)=0 & \quad &\mbox{in }\B_\tau^+,\\
&\mathrm{div}(t^{1-2\sigma} \nabla V)=0 & \quad &\mbox{in }\B_\tau^+,\\
&\frac{\pa U}{\pa \nu^\sigma} = \alpha_1U^{\frac{n+2\sigma}{n-2\sigma}}+\beta U^{\frac{2\sigma}{n-2\sigma}}V^{\frac{n}{n-2\sigma}} &\quad& \mbox{on }\pa' \B_\tau,\\
&\frac{\pa V}{\pa \nu^\sigma} = \alpha_2V^{\frac{n+2\sigma}{n-2\sigma}}+\beta V^{\frac{2\sigma}{n-2\sigma}}U^{\frac{n}{n-2\sigma}} &\quad& \mbox{on }\pa' \B_\tau,
\end{aligned}
\end{cases}
\]
with the help of  Proposition \ref{Harnack inequality}, which ensures that both $u$ and $v$ are H\"{o}lder continuous at $0$.
\end{proof}

\begin{proof}[Proof of Theorem \ref{thm:a}]
It follows from Theorem \ref{thm2}, \ref{14} and \ref{15}.
\end{proof}

\section{Appendix}\label{F5}
\begin{prop}\label{QL}
There exists a constant $C_3$ depending on $n$, $\beta$,  $\sigma$, $x$ such that
\[
I_2\leq C_3\int_{B_{\mu}(x)\backslash B_{\lda}(x)}\left((u_{x,\lda }-u)^++(v_{x,\lda }-v)^+\right)(u_{x,\lda }-u)^+,
\]
where $I_2$ is defined in \eqref{4}.
\end{prop}
\begin{proof}
Notice that for all $ y\in B_{\mu}(x)\backslash B_{\lda}(x)$,
\begin{equation}\label{FUDE}
u_{x,\lda }(y)^{\frac{4\sigma-n}{n-2\sigma}}v_{x,\lda }(y)^{\frac{n}{n-2\sigma}}=\left(\frac{\lda }{|y-x|}\right)^{4\sigma}u(y_{\lda })^{\frac{4\sigma-n}{n-2\sigma}}v(y_{\lda })^{\frac{n}{n-2\sigma}}
\leq u(y_{\lda})^{\frac{4\sigma-n}{n-2\sigma}}v(y_{\lda })^{\frac{n}{n-2\sigma}},
\end{equation}
where $y_{\lda}:=x+\frac{\lda^2(y-x)}{|y-x|^2}$. Combining $y_{\lda }\in B_{\lda}(x)\subset\overline{B}_{10}(x)$
and $u\in C^{2}(\R^{n})$, we obtain that there exist positive constants $C_{21}$, $C_{22}$ depending on $x$, such that for $y_{\lda }\in \overline{B}_{\lda}(x)\subset\overline{B}_{10}(x)$,
\[
C_{21}\leq u(y_{\lda }),v(y_{\lda })\leq C_{22}.
\]
A consequence of \eqref{FUDE} is that for all $y_{\lda }\in \overline{B}_{\lda}(x)$, there exists a positve constant $C_2$ depending on $x$, such that
\begin{equation}\label{RM}
u_{x,\lda }(y)^{\frac{2\sigma}{n-2\sigma}}v_{x,\lda }(y)^{\frac{2\sigma}{n-2\sigma}},\ \ u_{x,\lda }(y)^{\frac{4\sigma-n}{n-2\sigma}}v_{x,\lda }(y)^{\frac{n}{n-2\sigma}}\leq C_2.
\end{equation}
We only need to consider $u\leq u_{x,\lda }$. If $v_{x,\lda }\leq v$, it follows that
\begin{equation*}
\begin{split}
u(y)^{\frac{2\sigma}{n-2\sigma}}v(y)^{\frac{n}{n-2\sigma}}\geq& u(y)^{\frac{2\sigma}{n-2\sigma}}v(y)_{x,\lda }^{\frac{n}{n-2\sigma}}\left( \frac{u(y)}{u_{x,\lda }(y)} \right)^{\frac{n}{n-2\sigma}}\\
=&u_{x,\lda }(y)^{\frac{2\sigma}{n-2\sigma}}\left( \frac{u(y)}{u_{x,\lda }(y)} \right)^{\frac{2\sigma}{n-2\sigma}}v_{x,\lda }(y)^{\frac{n}{n-2\sigma}}\left( \frac{u(y)}{u_{x,\lda }(y)} \right)^{\frac{n}{n-2\sigma}}\\
=&u_{x,\lda }(y)^{\frac{2\sigma}{n-2\sigma}}v_{x,\lda }(y)^{\frac{n}{n-2\sigma}}\left( \frac{u(y)}{u_{x,\lda }(y)} \right)^{\frac{n+2\sigma}{n-2\sigma}},
\end{split}
\end{equation*}
as a result,
\begin{equation*}\label{DI1}
\begin{split}
&u_{x,\lda }(y)^{\frac{2\sigma}{n-2\sigma}}v_{x,\lda }(y)^{\frac{n}{n-2\sigma}}- u(y)^{\frac{2\sigma}{n-2\sigma}}v(y)^{\frac{n}{n-2\sigma}}\\
\leq& u_{x,\lda }(y)^{\frac{2\sigma}{n-2\sigma}}v_{x,\lda }(y)^{\frac{n}{n-2\sigma}}-u_{x,\lda }(y)^{\frac{2\sigma}{n-2\sigma}}v_{x,\lda }(y)^{\frac{n}{n-2\sigma}}\left( \frac{u(y)}{u_{x,\lda }(y)} \right)^{\frac{n+2\sigma}{n-2\sigma}}\\
=&u_{x,\lda }(y)^{\frac{2\sigma}{n-2\sigma}}v_{x,\lda }v^{\frac{n}{n-2\sigma}}\left(1-\left( \frac{u(y)}{u_{x,\lda }(y)} \right)^{\frac{n+2\sigma}{n-2\sigma}}\right)\\
\leq&\frac{n+2\sigma}{n-2\sigma}u_{x,\lda }(y)^{\frac{2\sigma}{n-2\sigma}}v_{x,\lda }(y)^{\frac{n}{n-2\sigma}}\left(1-\frac{u(y)}{u_{x,\lda }(y)} \right)\\
=&\frac{n+2\sigma}{n-2\sigma}u_{x,\lda (y)}^{\frac{4\sigma-n}{n-2\sigma}}v_{x,\lda }(y)^{\frac{n}{n-2\sigma}}\left(u_{x,\lda }(y)-u (y)\right).
\end{split}
\end{equation*}
If $v\leq v_{x,\lda }$, we can get that
\begin{equation*}
\begin{split}
&u(y)^{\frac{2\sigma}{n-2\sigma}}v(y)^{\frac{n}{n-2\sigma}}\geq u(y)^{\frac{2\sigma}{n-2\sigma}}\left( \frac{v(y)}{v_{x,\lda }(y)} \right)^{\frac{2\sigma}{n-2\sigma}}v(y)^{\frac{n}{n-2\sigma}}\left( \frac{u(y)}{u_{x,\lda }(y)} \right)^{\frac{n}{n-2\sigma}}\\
=&u_{x,\lda }(y)^{\frac{2\sigma}{n-2\sigma}}\left( \frac{u(y)}{u_{x,\lda }(y)} \right)^{\frac{2\sigma}{n-2\sigma}}\left( \frac{v(y)}{v_{x,\lda }(y)} \right)^{\frac{2\sigma}{n-2\sigma}}v_{x,\lda }(y)^{\frac{n}{n-2\sigma}}\left( \frac{v(y)}{v_{x,\lda }(y)} \right)^{\frac{n}{n-2\sigma}}\left( \frac{u(y)}{u_{x,\lda }(y)} \right)^{\frac{n}{n-2\sigma}}\\
=&u_{x,\lda }(y)^{\frac{2\sigma}{n-2\sigma}}v_{x,\lda }(y)^{\frac{n}{n-2\sigma}}\left( \frac{u(y)}{u_{x,\lda }(y)} \right)^{\frac{n+2\sigma}{n-2\sigma}}\left( \frac{v(y)}{v_{x,\lda }(y)} \right)^{\frac{n+2\sigma}{n-2\sigma}},
\end{split}
\end{equation*}
which implies that
\begin{equation*}\label{DI2}
\begin{split}
&u_{x,\lda }(y)^{\frac{2\sigma}{n-2\sigma}}v_{x,\lda }(y)^{\frac{n}{n-2\sigma}}- u(y)^{\frac{2\sigma}{n-2\sigma}}v(y)^{\frac{n}{n-2\sigma}}\\
\leq& u_{x,\lda }(y)^{\frac{2\sigma}{n-2\sigma}}v_{x,\lda }(y)^{\frac{n}{n-2\sigma}}-u_{x,\lda }(y)^{\frac{2\sigma}{n-2\sigma}}v_{x,\lda }(y)^{\frac{n}{n-2\sigma}}\left( \frac{u(y)}{u_{x,\lda }(y)} \right)^{\frac{n+2\sigma}{n-2\sigma}}\left( \frac{v(y)}{v_{x,\lda }(y)} \right)^{\frac{n+2\sigma}{n-2\sigma}}\\
=&u_{x,\lda }(y)^{\frac{2\sigma}{n-2\sigma}}v_{x,\lda }(y)^{\frac{n}{n-2\sigma}}\left(1-\left( \frac{u(y)}{u_{x,\lda }(y)} \right)^{\frac{n+2\sigma}{n-2\sigma}}\left( \frac{v(y)}{v_{x,\lda }(y)} \right)^{\frac{n+2\sigma}{n-2\sigma}}\right)\\
\leq&\frac{n+2\sigma}{n-2\sigma}u_{x,\lda }(y)^{\frac{2\sigma}{n-2\sigma}}v_{x,\lda }(y)^{\frac{n}{n-2\sigma}}\left(\left(1-\frac{u(y)}{u_{x,\lda }(y)}\right) +\left(1-\frac{v(y)}{v_{x,\lda }(y)}\right)\right)\\
=&\frac{n+2\sigma}{n-2\sigma}u_{x,\lda }(y)^{\frac{4\sigma-n}{n-2\sigma}}v_{x,\lda }(y)^{\frac{n}{n-2\sigma}}\left(u_{x,\lda }(y)-u(y) \right)\\
&+\frac{n+2\sigma}{n-2\sigma}u_{x,\lda }(y)^{\frac{2\sigma}{n-2\sigma}}v_{x,\lda }(y)^{\frac{2\sigma}{n-2\sigma}}\left(v_{x,\lda }(y)-v (y)\right).
\end{split}
\end{equation*}
Consequently,
\[
I_2\leq C_3\int_{B_{\mu}(x)\backslash B_{\lda}(x)}\left((u_{x,\lda }-u)^++(v_{x,\lda }-v)^+\right)(u_{x,\lda }-u)^+.
\]
Here $C_3$ is a constant depending on $n$, $\beta$,  $\sigma$, $x$.
\end{proof}
\subsection{}
\begin{prop}\label{PP}\rm{\cite[Proposition 4.1]{LB}}
Let $x\in \R^{n}$, $\mu>0$, $x_0\in \R^{n}\backslash B_{\mu}(x)$ and  $X=(x,0)$. If $U\in C^2(\R^{n+1}_+\backslash\B^{+}_{\mu}(X))$ is a nonnegative solution of
\[
\begin{cases}
\begin{aligned}
&\mathrm{div}(t^{1-2\sigma} \nabla U)\le 0 & \quad& \mbox{{\rm in} }\ \R^{n+1}_+\backslash\B^{+}_{\mu}(X),\\
&\frac{\partial U}{\partial \nu^\sigma} \ge 0 &\quad& \mbox{{\rm on} }\ \partial'(\R^{n+1}_+\backslash\B^{+}_{\mu}(X))\backslash\{X_0\},
\end{aligned}
\end{cases}
\]
and $\liminf_{\xi\to \infty}U(\xi)\geq0$, then
\[
U(\xi)\geq \left(\frac{\mu}{|\xi-X|}\right)^{n-2\sigma}\inf_{\partial'' \B^{+}_{\mu}(X)} U\quad\quad \mbox{{\rm in} }\ \  \R^{n+1}_+\backslash\B^{+}_{\mu}(X).
\]
\end{prop}
\subsection{Some useful propositions}

\begin{prop}\label{CHANG1}\rm{\cite[Lemma 11.1]{LZ}}
Let $f\in C^{1}(\R^{n})$, $n\geq 1$, $\gamma> 0$. Suppose that for every $x\in \R^{n}$, there exists $\lambda(x)>0$ such that
\[
\left(\frac{\lambda(x) }{|y-x|}\right)^{\gamma}f\left(x+\frac{\lambda(x)^2(y-x)}{|y-x|^2}\right)=f(y)\quad\quad\mbox{in }\ \ \R^{n}\setminus \{x\}.
\]
Then for some $a\geq 0$, $d>0$, $\overline{x}\in \R^{n}$,
\[
f(x)=\pm\left(\frac{a}{d+|x-\overline{x}|^2}\right)^{\frac{\gamma}{2}}.
\]
\end{prop}
\begin{prop}\label{CHANG2}\rm{\cite[Lemma 11.2]{LZ}}
Let $f\in C^{1}(\R^{n})$, $n\geq 2$, $\gamma> 0$. Assume that $x\in \R^{n}$, if for all $\lambda>0$,
\[
\left(\frac{\lambda }{|y-x|}\right)^{\gamma}f\left(x+\frac{\lambda^2(y-x)}{|y-x|^2}\right)\leq f(y)\quad\quad \mbox{in }\ \ \R^{n}\backslash B_{\lambda}(x).
\]
Then
\[
f(x)=constant.
\]
\end{prop}
\begin{prop}\label{CHANG2}\rm{\cite[Lemma 11.3]{LZ}}
Let $f\in C^{1}(\R_+^{n+1})$, $n\geq 1$, $\gamma> 0$. Assume that $X\in \partial\R_+^{n+1}$, if for all $\lambda>0$,
\be
\left(\frac{\lambda }{|\xi-X|}\right)^{\gamma}f\left(X+\frac{\lambda^2(\xi-X)}{|\xi-X|^2}\right)\leq f(\xi)\quad \mbox{in}\ \ \R_+^{n+1}\backslash \B^+_{\lambda}(x).
\ee
Then
\[
f(Y)=f(y',t)=f(0,t)\quad\mbox{in}\ \ \R_+^{n+1}.
\]
\end{prop}
\begin{prop}\label{prop:guji1}\rm {\cite[Proposition 2.9]{Silvestre}}
Let
\be
(-\Delta)^{\sigma} u=g(x)\quad \mbox{in }\ B_R.
\ee
Assume $u$ and $g\in L^{\infty}(B_R)$, then there exists $\gamma \in (0,1)$, such that $u\in C^{\gamma}(B_{R/2})$. Moreover,
\be
\|u\|_{C^{\gamma}(B_{R/2})}\leq C\left(\|u\|_{L^{\infty}(B_R)}+\|g\|_{L^{\infty}(B_R)}\right),
\ee
for a positive constant $C$ depending only on $n$, $\sigma$, $R$.
\end{prop}

\begin{prop}\label{prop:guji}\rm{\cite[Theorem 2.1]{JLX}}
 Let
\be
(-\Delta)^{\sigma} u=g(x)\quad \mbox{in }\ B_R.
\ee
Assume $u\in L^{\infty}(B_R)$ and $g\in C^{\gamma}(B_R)$, $\gamma>0$  and  $2\sigma+\gamma$ is not an integer. Then $u\in C^{2\sigma+\gamma}(B_{R/2})$. Moreover,
\be
\|u\|_{C^{2\sigma+\gamma}(B_{R/2})}\leq C\left(\|u\|_{L^{\infty}(B_R)}+\|g\|_{C^{\gamma}(B_{3R/4})}\right),
\ee
for a positive constant $C$ depending only on $n$, $\sigma$, $\gamma$, $R$.
\end{prop}
\begin{prop}\label{Tidu}\rm{\cite[Lemma 4.5]{XY}}
Let $g\in C^{\alpha}(B_R)$ for some $\alpha\in(0,1)$ and $U\in W(t^{1-2\sigma}, \B^{+}_{R})\cap L^{\infty}(\B^+_R)$ be a weak solution  of
\[
\begin{cases}
\begin{aligned}
&\mathrm{div}(t^{1-2\sigma} \nabla U)=0 & \quad &\mbox{in }\B^{+}_R,\\
&\frac{\pa U}{\pa \nu^\sigma} = g(x)&\quad &\mbox{on }\pa'\B^{+}_{R},
\end{aligned}
\end{cases}
\]
then there exists $\gamma \in (0,1)$ depending only on $n$, $\sigma$ and $\alpha$, such that $U\in C^{\gamma}(\overline{\B^{+}_{R/2}})$ and $t^{1-2\sigma}\partial_t U\in C^{\gamma}(\overline{\B^{+}_{R/2}})$.
Furthermore, there exist constants $C_1$ and $C_2$ depending only on $n$, $\sigma$, $\alpha$, $R$, $\|U\|_{L^\infty(\B^{+}_R)}$ and also on $\|g\|_{L^\infty(B_R)}$  for $C_1$, $\|g\|_{C^{\alpha}(B_R)}$
for $C_2$ such that
\[
\|U\|_{C^{\gamma}(\overline{\B^{+}_{R/2}})}\leq C_1,
\]
and
\[
\|t^{1-2\sigma}\partial_t U\|_{C^\gamma (\overline{\B^{+}_{R/2}})}\leq C_2.
\]
\end{prop}
\begin{prop}\label{Tidu1}\rm{\cite[Proposition 2.13]{JLX}}
Let $g\in C^{1}(B_R)$ and $U\in W(t^{1-2\sigma}, \B^{+}_{R})$ be a weak solution  of
\[
\begin{cases}
\begin{aligned}
&\mathrm{div}(t^{1-2\sigma} \nabla U)=0 & \quad &\mbox{in }\B^{+}_R,\\
&\frac{\pa U}{\pa \nu^\sigma}= g(x)&\quad &\mbox{on }\pa'\B^{+}_{R},
\end{aligned}
\end{cases}
\]
then there exists a positive constant $C$ depending only on $n$, $\sigma$,  $R$
such that
\[
\|\nabla U\|_{L^{\infty}(\B^{+}_{R/2})}\leq C(\|U\|_{L^2(t^{1-2\sigma},\B^{+}_R)}+\|g\|_{C^1(B_R)}).
\]
\end{prop}
\begin{prop}\label{prop:liminf}\rm {\cite[Proposition 3.1]{JLX}}
Suppose that $U\in C^{2}(\B_R^+)\cap C^{1}(\overline{\B_R^+}\setminus\{0\})$  and $U>0$ in $\B_R^+ \cup \partial'\B_R^+\setminus \{0\} $ is a  solution of
\[
\begin{cases}
\begin{aligned}
&\mathrm{div}(t^{1-2\sigma} \nabla U)\le 0 & \quad& \mbox{in }\ \ \B_R^+,\\
&\frac{\partial U}{\partial \nu^\sigma}  \ge 0 &\quad& \mbox{on }\ \pa'\B^{+}_{R}\setminus \{0\},
\end{aligned}
\end{cases}
\]
then
\[
\liminf_{X\to 0}U(X)>0.
\]
\end{prop}

\begin{prop}\label{A.2}\rm{\cite[Lemma A.2]{JLX}}
Suppose that $U\in W^{1£¬2}(t^{1-2\sigma}, \B^+_{R})$, $U\geq 0$ on $\partial''\B^+_{R}$ is a weak supsolution of
\[
\begin{cases}
\begin{aligned}
&\mathrm{div}(t^{1-2\sigma} \nabla U)=0 & \quad &\mbox{in }\B^{+}_R,\\
&\frac{\pa U}{\pa \nu^\sigma} =a(x)U&\quad &\mbox{on }\pa'\B^{+}_{R},
\end{aligned}
\end{cases}
\]
if there exists $\varepsilon=\varepsilon(n,\sigma)$ such that for all $|a(x)|\leq \varepsilon|x|^{-2\sigma}$,
then
\[
U\geq 0\quad\quad\mbox{in }\ \ \B^{+}_R.
\]
\end{prop}
{\bf Acknowledgements}
We would like to express our deep thanks to Professor Jingang Xiong for useful discussions on the subject of this paper.

{}
\end{document}